\documentclass[12pt]{article}
\usepackage{amsmath}
\usepackage{amsthm}
\usepackage{amssymb}
 
\numberwithin{equation}{section}
\newcommand{\dcb}{\begin{array}{lll}}
\newcommand{\dce}{\end{array}}

\newtheorem{theo}{Theorem}[section]
\newtheorem{prop}[theo]{Proposition}

\newtheorem{rem}[theo]{Remark}

\theoremstyle{definition}

\newtheorem{ass}[theo]{Assumption}

\def \E{I\!\!E}
\def \P{I\!\!P}

\def \RR{I\!\!R}

\def \FF{I\!\!F}
\def \GG{C\!\!\!\!G}

\def \RR{I\!\!R}
\def \ZZ{Z\!\!\!Z}
\def \FF{I\!\!F}

\def \GG{C\!\!\!\!G}

\def \ge{\varepsilon}

\newlength{\breite}
\title{Polynomial deviation bounds for recurrent Harris processes having general state space}

\author{Eva 
{\sc L\"ocherbach}\footnote{CNRS UMR 8088, D\'epartement de Math\'ematiques, Universit\'e de Cergy-Pontoise,
95 000 CERGY-PONTOISE,  France. E-mail: {\tt 
eva.loecherbach@u-cergy.fr}} and Dasha {\sc Loukianova}\footnote{D\'epartement de Math\'ematiques, Universit\'e d'Evry-Val d'Essonne, Bd Fran\c{c}ois Mitterrand, 91025 Evry Cedex, France. E-mail: {\tt 
 dasha.loukianova@univ-evry.fr
} }}
\begin{document}
\maketitle
\def\abstractname{Abstract}
\begin{abstract}
Consider a strong Markov process in continuous time, taking values in some Polish state space. Recently, Douc, Fort and Guillin (2009) introduced verifiable conditions in terms of a supermartingale property  implying an explicit control of modulated moments of hitting times. We show how this control can be translated into a control of polynomial moments of abstract regeneration times 
which are obtained by using the regeneration method of Nummelin, extended to the time-continuous context. 

As a consequence, if a $p-$th moment of the regeneration times exists, we obtain non asymptotic deviation bounds of the form 
$$P_{\nu}\left (\left|\frac1t\int_0^tf(X_s)ds-\mu(f)\right|\geq\ge\right)\leq K(p)\frac1{t^{p- 1}}\frac 1{\ge^{2(p-1)}}\|f\|_\infty^{2(p-1)} , p \geq 2. $$
Here, $f$ is a bounded function and $\mu$ is the invariant measure of the process. 
We give several examples, including elliptic stochastic differential equations and stochastic differential equations driven by a jump noise. 
\end{abstract}

\\

{\it MSC 2000}  : 60 J 55, 60 J 35, 60 F 10, 62 M 05

\section{Introduction}
Let $X$ be a positive Harris recurrent strong Markov process in continuous time, having invariant 
probability measure $\mu .$  From the Ergodic Theorem we know that for all $x\in\RR $, 
$f\in L^1(\mu)$ and $\ge>0$ 
 \begin{equation}\label{ergthm}
  P_x\left (\left|\frac1t\int_0^tf(X_s)ds-\mu(f)\right|\geq\ge\right)\to 0
  \end{equation}
  as $t$ goes to infinity. The purpose of this paper is to 
  establish the rate of convergence in (\ref{ergthm}),
  for bounded functions $f.$ In the existing literature, mainly the case of exponential rate of convergence (exponential 
ergodicity) has been considered. But recently, there has been growing interest in studying other possible rates such as sub-geometric or
polynomial rates. We follow this direction and study in this paper the case when the rate of convergence 
in (\ref{ergthm}) is polynomial. More precisely we use the so-called regeneration method and show that if a certain regeneration time admits a $p-$th moment, then we obtain
non asymptotic deviation bounds of the form 
\begin{equation}\label{erginequality}
P_{x }\left (\left|\frac1t\int_0^tf(X_s)ds-\mu(f)\right|\geq\ge\right)\leq K(p,x )\frac1{t^{p- 1}}\frac 1{\ge^{2(p-1)}}\|f\|_\infty^{2(p-1)}, p \geq 2.
\end{equation}
Here, $f$ is a bounded function and $\mu$ is the invariant measure of the process. Such a bound  is of major importance for many applications, for example non asymptotic problems for statistics of diffusions, concentration for particular approximations of granular media equations, and many other examples.

Let us give some comments on the history of the problem and compare our result with known results on deviation inequalities for Markov processes. In the context of Markov chains, Cl\'emen\c{c}on (2001) and Bertail and Cl\'emen\c{c}on (2009) 
have obtained bounds in (\ref{ergthm}) which are exponential in time, using the regeneration method of Nummelin. They work under the conditions of geometric (exponential) ergodicity and stationarity, and within the space of bounded functions. Our work is close to this in spirit, since we use the regeneration method as well (however, we use it in a more complicated framework since we work in continuous time). Compared to their work we do not need to assume
stationarity, our results hold for any starting point $x$ or any starting measure provided it integrates the $p-$th moment of the regeneration time. Moreover, we weaken the assumption of exponential ergodicity to polynomial ergodicity. Still in the discrete framework of Markov chains let us also mention the work by Adamczak (2009) who derives, using completely different techniques, concentration inequalities for empirical processes of Markov chains, in the regime of exponential ergodicity. Finally, Chazottes, Collet, K\"ulske and Redig (2007) obtain concentration inequalities for finite valued random fields on $\ZZ^d $  via coupling both in the exponential and the sub-exponential regime. For their purposes, the finiteness of the phase space is crucial. 

All above mentioned results hold either in discrete time or in discrete state space, and this is not what we are interested in. In this paper we concentrate on the framework of continuous time and general state space. For continuous time Markov processes there is a huge literature on the subject, and most of the results are based on functional inequalities and/or perturbation techniques which are used to obtain non-asymptotic bounds in (\ref{ergthm}). As a matter of fact, in contrary to our approach, most of these papers deal with the stationary case only or with the case when the initial law of the process
is absolutely continuous with respect to the invariant measure, having a square integrable density. Wu (2000) uses the Lumer-Phillips theorem in order to derive non-asymptotic deviation bounds which are expressed in terms of a large deviations rate function. He works under the assumption that the initial law of the process is absolutely continuous with
respect to the invariant measure. Based on this, Cattiaux and Guillin (2008) use functional inequalities like the Poincar\'e inequality in order to derive an exponential deviation bound; they work under the assumption of a spectral gap and with bounded functions. A small paragraph in Cattiaux and Guillin (2008) is devoted to the polynomial regime as well, under an assumption imposing polynomial decay of the $\alpha -$mixing coefficient of the process, but the rate which is obtained is not optimal. In the same spirit, let us cite Guillin, L\'eonard, Wu and Yao (2009) who work in the space of Lipschitz functions under the assumption of a spectral gap. For bounded functions, they obtain a Hoeffding type inequality. Finally, Lezaud (2001) uses Kato's theory of perturbation of operators, still in the exponential regime. Let us also mention that in a completely different setting and having different applications in mind, Pal (2011) establishes concentration inequalities for diffusion laws on the path space $C([0, \infty ) ),$ using quadratic transportation cost inequalities. He 
studies concentration around the median of the distribution, in the exponential regime, for Lipschitz functions on the path space with respect to the uniform norm.  

In contrast to most of the above mentioned papers, we do not assume exponential ergodicity, nor the existence of a spectral gap nor stationarity. We do not need to assume that the process is $\mu -$symmetric. The method we use is the so-called regeneration method. It appeals to the condition of integrability of regeneration times. Let us describe briefly what is the idea of regeneration times. In the easiest situation where the process $X$ has a 
recurrent point $x_0 ,$ we may introduce 
a sequence of stopping times $R_n ,$ called {\it regeneration times}, such that 
\begin{enumerate}
\item
For all $n,$ $R_n < \infty, $  $R_{n+1} = R_n + R_1 \circ \theta_{R_n},$ $R_n \to \infty $ as $n \to \infty .$ (Here, $ \theta $ denotes the shift operator.)
\item
For all $n,$ $X_{R_n } = x_0 .$
\item
For all $n ,$ the process $ (X_{R_n + t})_{t \geq 0 } $ is independent of ${\cal F}_{R_n} .$
\end{enumerate} 
In this case, paths of the process can be decomposed into i.i.d. excursions $[R_i, R_{i+1}[, 
i=1, 2 , \ldots, $ plus an initial segment $[0, R_1],$ and then limit theorems follow immediately from the strong law of large numbers. 

In general, recurrent points exist only in one-dimensional models. For one-dimensional recurrent diffusions 
it has been shown in L\"ocherbach, \\Loukianova and Loukianov (2011) that,   
if for some $p> 1$  the $p$-th moment of the regeneration time exists, then (\ref{erginequality}) holds.

For general multidimensional Harris recurrent processes, 
there is no direct way of defining regeneration times. However, there is a well-known
method of introducing regeneration times artificially, which is known as method of ``Nummelin splitting''
in the case of Markov chains and which has been extended to the case of processes in continuous
time by L\"ocherbach and Loukianova (2008). This method consists of constructing a bigger process $Z = (Z^1, Z^2, Z^3)$ 
taking values in $ E \times [0,1 ] \times E,$ along a sequence
of jump times $0 = T_0 < T_1 < \ldots < T_n < \ldots , $ such that 
\begin{enumerate} 
\item
$Z^1 $ is a copy of the original process $X$, and 
the $T_n$ are arrival times of a rate-$1-$Poisson process, independent of $Z^1 .$
\item
On each time interval $[T_n, T_{n+1}[,$ $Z^2 $ and $Z^3$ are constant. 
\item
The sequence $ (Z^3_{T_n})_n $ is a copy of the resolvent chain $ X_{T_{n+1}} $ (the process $X$ observed
after independent exponential times). 
\item
The sequence $(Z^2_{T_n} )_n $ is a copy of independent random variables, which are uniformly distributed on $[0,1].$ 
\end{enumerate}
The three co-ordinates and the sequence of jump times $(T_n)_n$ are
constructed in a {\it coupled} way, inspired by the splitting technique of Nummelin (1978) and Athreya and Ney (1978) in discrete time. We recall the whole construction  
in Section \ref{section:nummelin}. The main point of this construction is that there exist a measurable set $C$ having $\mu (C) > 0 $ 
($C$ will be a {\it petite set} in the Meyn-Tweedie terminology) and a parameter $\alpha \in ] 0,1] $ such that the 
successive visits of $Z_{T_n} $ to $ C \times [0, \alpha ] \times E $ induce regeneration times for the process $Z.$ 

To resume, for any Harris recurrent Markov process $X, $ the following holds true: the process $X$ can be 
embedded as first co-ordinate into a new Markov process $Z .$ This new process $Z$ possesses regeneration times. 
These regeneration times are closely related to the hitting times of a certain petite set $C$, or in other words: 
the moments of regeneration times are closely related to hitting time moments. Once we have a $p-$th moment for
the regeneration times, we obtain a control on the speed of convergence in the ergodic Theorem and 
(\ref{erginequality}) holds true. 

Note that different coupling techniques in spirit of the so-called Doeblin- or Dobrushin-coupling have been considered in the literature, for example
in the case of diffusions by Veretennikov (1997) and (2004), and for L\'evy-noise driven solutions of SDE's by Kulik (2009). 
These couplings are more specific to the concrete models the authors are interested in -- the coupling technique presented in this
paper has the advantage of being completely general, as far as Harris recurrent processes are concerned. 

Once the coupling is constructed, it remains to establish sufficient conditions on the generator of the process ensuring that $p-$th moments for
regeneration times exist. These conditions are inspired by a
recent work of Douc et al. (2009) on sub-geometric rates of convergence for strong Markov processes. In this work, the authors 
introduce a drift condition towards a closed petite set in the spirit of a condition of existence of a Lyapunov 
function. This condition provides an upper bound on the control of sub-geometric or polynomial moments of 
hitting times where the dependence on the starting point is precisely given.  
The drift condition also provides a verifiable condition ensuring positive Harris recurrence of 
the process. We recall these results in Section \ref{section:drift}. 
Section \ref{section:nummelin} is devoted to give a self-contained description of 
the state of the art concerning the regeneration or Nummelin-splitting-method in the
multidimensional case. 
Section \ref{sec:polmoment} provides a link between the two approaches ``Drift Condition''
of Douc et al. (2009) and ``Nummelin splitting''.  We show that the drift condition of Douc et al. (2009) 
provides an upper bound on the regeneration times introduced according to the method
of Nummelin splitting. More precisely, we show in Theorem \ref{theo:polymoment} that certain polynomial moments up to a precise order are bounded - the bound on the order being determined by the Lyapunov condition. The dependence upon the starting point is controlled by the Lyapunov function as usual.  So even though the moments of regeneration times can not be
explicitly calculated, we get at least upper bounds in the rate of convergence in (\ref{ergthm}). 
As a main application of this result, in Section \ref{lastsection} we state and give the proof of the deviation inequality (\ref{erginequality}). 
Section 6 is devoted to some examples: multi-dimensional diffusions and SDE's driven by a jump noise that are treated in the spirit of a recent work of Kulik (2009). We close the paper with an appendix which recalls the Fuk-Nagaev inequality in the framework needed in Section \ref{lastsection}.

\vskip0.5cm \noindent 
{\bf Acknowledgements.}   
Eva L\"ocherbach has been partially supported by an ANR project: Ce travail a b\'en\'efici\'e d'une aide de l'Agence Nationale de la Recherche
portant la r\'ef\'erence ANR-08-BLAN-0220-01.
\section{Drift-condition, Harris-recurrence and modulated moments}\label{section:drift}

Consider a probability space $(\Omega, {\cal A}, (P_x)_x) .$ Let $X
= (X_t)_{t \geq 0 }$ be a process defined on $(\Omega , {\cal A},
(P_x)_x) $ which is strong Markov, taking values in a locally
compact Polish space $(E,{\cal E}), $ with c\`adl\`ag paths. $ (P_x)_{x \in E}$ is a collection 
of probability measures on $(\Omega, {\cal A})$ such that $X_0 =
x $ $P_x -$almost surely. We write 
$(P_t)_t $ for the transition semigroup of $X .$  Moreover, we shall write
$({\cal F}_t)_t$ for the filtration generated by the process.

Throughout this paper, we impose the following condition on the transition semigroup $(P_t)_t $ of $X .$

\begin{ass}\label{regularitypt}
There exists a sigma-finite positive measure $\Lambda $ on $(E,
{\cal E}) $ such that for every $t > 0,$ $P_t (x,dy ) = p_t(x,y)
\Lambda (dy) ,$ where $(t,x,y) \mapsto p_t (x,y) $ is jointly
measurable.
\end{ass}

We are seeking for conditions ensuring that the process $X$ is recurrent in the
sense of Harris. The most popular conditions for Harris-recurrence are drift conditions or more generally conditions in terms of
a supermartingale property for a functional of the Markov process. We follow Douc et al.~(2009) and impose a drift condition towards a closed petite
set $B $ which implies the Harris recurrence of the process. 
Recall that a set $B \in {\cal E}$ is {\it petite} if there exists a probability measure 
$a$ on ${\cal B} ( \RR_+ ) $ and a measure $\nu_a $ on $(E,{\cal E}) $
such that 
\begin{equation}\label{eq:petite}
 \int_0^\infty P_t (x , dy)a( dt) \geq 1_B (x) \nu_a (dy) . 
\end{equation}

\begin{ass}\label{ass:driftcond}
There exists a closed petite set $B$, a continuous function $V : E \to [1, \infty [ ,$ 
an increasing differentiable concave positive function $\Phi : [1, \infty ) \to (0, \infty ) $
and a constant $b < \infty $ such that for any $s \geq 0,$ $x \in E,$ 
\begin{equation}\label{eq:driftcond}
E_x ( V(X_s)) + E_x \left( \int_0^s \Phi \circ V(X_u) du \right) \le 
V(x) + b E_x \left( \int_0^s 1_B (X_u) du \right) .
\end{equation}
\end{ass} 

\begin{rem}
If $V \in {\cal D} ({\cal A}) $ belongs to the domain of the extended generator ${\cal A}$ of the process
$X ,$ then Theorem 3.4 of Douc et al.~(2009) shows that 
\begin{equation}\label{nice}
{\cal A} V (x) \le - \Phi \circ V(x) + b 1_B (x) 
\end{equation}
implies the above Assumption \ref{ass:driftcond}. 
\end{rem}

By Proposition 3.1 of Douc et al.~(2009), we know that under Assumption 
\ref{ass:driftcond}, 
the process $X$ is positive 
recurrent in the sense of Harris. We write $\mu$ for its invariant probability measure.
Hence, for any
set $A \in {\cal E}$ such that $\mu (A) > 0, $ we have $ \lim\sup_{t \to
\infty} 1_A (X_t ) = 1 $ almost surely. In particular the process is $\mu -$irreducible.

Under Assumption \ref{ass:driftcond}, Douc et al.~(2009) give estimates on modulated moments
of hitting times. Modulated moments are expressions of the type 
$$ E_x \int_0^\tau r(s) f(X_s) ds ,$$
where $\tau $ is a certain hitting time, $r$ a rate function and $f$ any positive measurable function. 
Knowledge of the modulated moments permits to interpolate between the maximal rate of convergence
(taking $f \equiv 1$) and the maximal shape of functions $f$ that can be taken in the
ergodic theorem (taking $r \equiv 1$). In the present paper we are interested in the maximal 
rate of convergence and hence we shall always take $f \equiv 1.$ 

For the function $\Phi$ of (\ref{eq:driftcond}) put 
\begin{equation}\label{eq:ratefunction}
 H_\Phi (u) = \int_1^u \frac{ds }{\Phi(s) },\; u \geq 1 , \; r_\Phi (s) = r(s) = \Phi \circ H_\Phi^{-1} (s) .
\end{equation}
We are interested in choices of the function $\Phi$ that yield a polynomial rate function $r.$ 
This is achieved by the choice $\Phi (v) = c v^\alpha $ for $0 \le \alpha < 1 $ giving rise to polynomial rate functions 
$$ r (s) \sim C s^{\frac{\alpha}{1-\alpha }}.$$
We suppose from now on that Assumption \ref{ass:driftcond} is satisfied with such a kind of function $\Phi (v) = c v^\alpha $ for $0 \le \alpha < 1 .$
The most important technical feature about the rate function that will be useful in the sequel is then the following sub-additivity property 
\begin{equation}\label{eq:subadditive}
 r (t+s) \le c ( r(t) + r(s)), 
\end{equation}  
for $ t, s \geq 0$ and $c$ a positive constant. We shall also use that  
$$ r(t+s) \le r(t) r(s) ,$$
for all $t,s \geq 0 . $ 

We are interested in regeneration time moments. We will see in Section 3 below that regeneration times are almost hitting 
times. Concerning hitting times, the following result is known in the literature. 
Fix $\delta > 0 $ and define for any closed set $A \in {\cal E}$ the delayed hitting time  
$$ \tau_A (\delta ) := \inf \{ t \geq \delta : X_t \in A \} .$$ 
Then Theorem 4.1 and Proposition 4.5, (ii) of Douc et al.~(2009) imply the following two statements.
Firstly, for the rate function $r$ of (\ref{eq:ratefunction}) and for the petite set $B$ of Assumption \ref{ass:driftcond},
\begin{equation}\label{eq:modmoment1}
E_x \int_0^{\tau_B (\delta)} r(s) ds \le V(x) - 1 + \frac{b}{\Phi(1)} \int_0^\delta r(s) ds .
\end{equation}
Second, for the rate function $r$ of (\ref{eq:ratefunction}) and for any closed set $A$ with $\mu (A) > 0 ,$ 
for any $\delta' > 0,$ 
\begin{equation}\label{eq:modmoment2}
E_x \int_0^{\tau_A (\delta')} r(s) ds \le c(A, \delta') \left[ V(x) - 1 + \frac{b}{\Phi(1)} \int_0^\delta r(s) ds \right] 
.
\end{equation}

\begin{rem}
Suppose that $ E = \RR $ and that the process $X$ has continuous trajectories. Fix a recurrent point $a \in \RR .$ Then 
we can choose $ A = [ a , \infty [ , $ if $x < a ,$ $A= ] - \infty , a ] ,$ if $x > a $ in (\ref{eq:modmoment2}) above. In this case, the successive visits 
$$ R_1 := \tau_{\{a\}} (\delta), \; R_{n +1 } := \inf \{ t \geq R_n + \delta : X_t = a \} $$
of the point $a$ are regeneration times of the process.
Hence, (\ref{eq:modmoment2}) gives a control of regeneration time moments in the one-dimensional case. 
\end{rem}

In the general multidimensional case, the times $\tau_A (\delta)$ do not define regeneration times any more.
In this case, at least in general, regeneration times can only be introduced in an artificial manner, using the technique of Nummelin splitting
in continuous time, as developed in L\"ocherbach and Loukianova (2008). 
However, the estimates (\ref{eq:modmoment1}) and (\ref{eq:modmoment2}) can be translated into 
bounds on moments of these new extended regeneration times of the process. This is the main issue of this paper and will be treated in 
section \ref{sec:polmoment} below. 

In the next section we recall the technique of Nummelin splitting and then give the bounds on moments of the regeneration times. But before doing this we first recall some known facts about 
modulated moments of the resolvent chain from Douc et al.~(2004).

\subsection{Modulated moments for the resolvent chain}
Observing the continuous time process after independent exponential times gives rise to the resolvent chain and allows
to use known results in discrete time instead of working with the continuous time process. This trick is quite often used
in the theory of processes in continuous time. 

Write $U^1 (x, dy) := \int_0^{\infty} e^{-t} P_t (x, dy) dt $
for the resolvent kernel associated to the process.
Introduce a sequence $ (\sigma_n)_{n \geq 1} $ of i.i.d.~$exp(1)$-waiting times, independent of the process $X$ itself. Let $T_0 = 0,$ $T_n = \sigma_1 + \ldots
+ \sigma_n$ and $\bar{X}_n = X_{T_n} .$  Then the chain $\bar{X}= (\bar{X}_n)_n$ is recurrent in the sense of Harris, having the same
invariant measure $\mu$ as the continuous time process,
and its one-step transition kernel is given by $U^1 (x, dy).$

Since $X$ is Harris, it 
can be shown (Revuz (1984), see also Proposition 6.7 of H\"opfner and
L\"ocherbach (2003)), that the
resolvent satisfies
\begin{equation}\label{minoration}
 U^1 (x, dy) \geq \alpha 1_C (x) \nu (dy) ,
\end{equation}
where $0 < \alpha \le  1, $ $\mu (C) > 0 $ and $\nu $ a probability measure equivalent to $\mu
(\cdot \cap C)$. The set $C$ is in general not the petite set of Assumption \ref{ass:driftcond}. It can be chosen to be compact.
In particular, (\ref{minoration}) implies that the resolvent chain is aperiodic. 

It is interesting to note that the drift condition (\ref{eq:driftcond}) on the process in continuous time implies a 
similar drift condition on the resolvent chain. More precisely, Theorem 4.9 of Douc et al.~(2009), item (ii), implies that under
Assumption \ref{ass:driftcond} the resolvent chain 
satisfies a drift condition as well, with a different petite set and different functions $\bar \Phi $ and $\bar V,$ but giving rise to the same rate function $r$ since 
$ \bar \Phi (t ( 1 + \Phi' (1))) \sim \Phi (t) $ for $t \to \infty .$ Moreover, 
$$ \| \bar V - V ( 1 + \Phi' (1)) \|_\infty < \infty .$$ 
Now for any measurable set $A$ with $\mu (A) > 0,$ write $\bar \tau_A := \inf \{ n \geq 1  : \bar X_n \in A \} .$
Then, by Douc et al.~(2004), proof of Theorem 2.8, second formula, 
\begin{equation}\label{eq:modmomentres}
 E_x \left[ \sum_{ k =0}^{\bar \tau_A -1 } r (k) \right] \le c_1 (A) \bar V (x) + c_2 (A) \le c_1 V(x) + c_2,
 \end{equation}
since $ \bar V(x) \le c_1 V(x) +c_2.$ 

After these preliminaries on resolvent chains we now turn to the description of the regeneration method in the case
of a general state space. 

\section{Nummelin splitting and regeneration times}\label{section:nummelin}
Regeneration times can be introduced for any Harris recurrent strong Markov process under
the Assumption \ref{regularitypt} -- without any further assumption. We make once more use of the resolvent
chain. Recall the definition of the resolvent kernel
$U^1 $ and the lower bound (\ref{minoration}) which holds under the only assumption of Harris recurrence:
$$ U^1 (x, dy) \geq \alpha 1_C (x) \nu (dy) ,$$
where $C$ is a fixed compact petite set with $\mu (C) > 0 .$ Note that since $\mu (C) > 0,$  (\ref{eq:modmoment2}) and (\ref{eq:modmomentres}) hold for the hitting time of this set $C.$  

\begin{rem}
Fort and Roberts (2005) and Douc et al.~(2009) impose quite 
systematically the condition of irreducibility of some skeleton chain, see e.g.~Theorem 3.2 and Theorem 3.3 of Douc et al.~(2009). 
This implies the existence of some $m$ such that $P_m$ satisfies 
$$ P_m (x, dy) \geq \alpha 1_C (x) \nu (dy) .$$
This condition is obviously stronger than (\ref{minoration}) 
and implies that the process is not only positive Harris recurrent but also ergodic, i.e. for all $x \in E,$ 
$$ || P_t (x,.) - \mu ||_{TV} \to 0 .$$
We do not impose this additional condition. 
\end{rem}

We now show how to construct regeneration times in
continuous time by using the technique of Nummelin splitting which has been introduced for Harris recurrent Markov chains in discrete time by
Nummelin (1978) and Athreya and Ney (1978). The idea is to define on an extension of the original space $(\Omega, {\cal A}, (P_x))$ a Markov process $Z = (Z_t)_{t \geq 0}= (Z_t^1, Z_t^2, Z_t^3)_{t \geq 0},$ taking values in $E \times [0, 1 ] \times E$ such that the times $T_n$ are jump times of the process and such that $((Z_t^1)_t, (T_n)_n)$ has the same distribution as $((X_t)_t, (T_n)_n).$ We recall the details of this construction from L\"ocherbach and Loukianova (2008). 

First of all, define the split kernel $Q ((x,u),dy).$ This is a transition kernel $Q ((x,u), dy) $ from $ E \times [0, 1 ] $ to $E $ defined by 
\begin{equation}\label{Q}
Q((x,u), dy) = \left\{
\begin{array}{ll}
\nu(dy) & \mbox{ if } (x,u) \in C \times [0, \alpha]\\
\frac{1}{1 - \alpha} \left( U^1 (x, dy) - \alpha \nu(dy) \right)  & \mbox{ if } (x,u) \in C \times ] \alpha , 1] \\
U^1 (x,dy) & \mbox{ if } x \notin C .
\end{array} \right. 
\end{equation}

\begin{rem}
This kernel is called split kernel since $\int_0^1 du Q((x,u), dy) = U^1 (x, dy).$ Thus $Q$ is a splitting of the resolvent kernel by means of the
additional ``colour'' $u.$
\end{rem}

Write $u^1 (x, x') := \int_0^\infty e^{-t} p_t (x, x') dt .$ We now show how to construct the process $Z$ recursively over time
intervals $[T_n, T_{n+1}[ , n \geq 0 .$ We start with some initial condition $Z^1_0  =  X_0 = x  ,$ $ Z_0^2 = u \in [0,1], $ $ Z_0^3 = x' \in E. $
Then inductively in $n \geq 0,$ on $Z_{T_n} = (x,u,x') :$
\begin{enumerate}
\item
Choose a new jump time $\sigma_{n+1} $ according to
$$e^{-t} \; \frac{p_t(x,x')}{u^1 (x, x') } \, dt  \mbox{ on } \RR_+,$$
where we define $0/0 := a/ \infty :=  1,$ for any $a \geq 0,$
and put $T_{n+1} := T_n + \sigma_{n+1}.$
\item
On $\{ \sigma_{n+1} = t \},$ put $Z_{T_n +s}^2 := u ,$ $Z_{T_n +s}^3 := x'  $ for all $0 \le s < t .$
\item
For every $s < t,$ choose
$$Z_{T_n + s}^1 \sim \frac{p_s(x,y) p_{t-s}(y , x')}{p_t (x,x') } \;  \Lambda (dy ) .$$
Choose $ Z_{T_n + s}^1 := x_0$ for some fixed point $x_0 \in E$ on $\{ p_t (x, x') = 0 \}.$
Moreover, given $Z_{T_n + s}^1 = y,$ on $s + u < t, $ choose
$$ Z^1_{T_n + s+ u } \sim \frac{p_u (y, y') p_{t-s-u}(y' , x')}{p_{t-s} (y, x') } \Lambda (dy') .$$
Again, on $\{ p_{t-s} (y, x')  = 0 \},$ choose $ Z^1_{T_n + s+ u } = x_0 .$
\item
At the jump time $T_{n+1},$ choose $Z^1_{T_{n+1}}  := Z^3_{T_n} = x' .  $ Choose $Z_{T_{n+1}}^2 $ independently of $Z_s, s < T_{n+1}  ,$ according to the uniform law $U.$ Finally, on $\{  Z^2_{T_{n+1}} = u' \},$ choose $ Z^3_{T_{n+1}} \sim Q((x',u'), dx'' ) .$
\end{enumerate}

Note that by construction, given the initial value of $Z$ at time $T_n,$ the evolution of the process $Z^1$  during $[T_n, T_{n+1}[$ does not depend on the chosen value of $Z^2_{T_n} .$

We will write $P_{\pi}$ for the measure related to $X$, under which $X$ starts from the initial measure $\pi(dx)$, and $\P_{\pi}$ for the measure related to $Z$, under which $Z$ starts from the initial measure $\pi(dx)\otimes U(du)\otimes Q((x,u),dy)$. Hence, $ \P_{x_0} $ denotes the
measure related to $Z$ under which $Z$ starts from the initial measure $ \delta_{x_0} (dx )  \otimes U(du)\otimes Q((x,u),dy)$. In the same spirit we denote
$E_{\pi}$ the expectation with respect to $P_{\pi}$  and $\E_{\pi}$ the expectation with respect to $\P_{\pi}$. Moreover, we shall write $\FF $ for  the filtration generated by $Z,$ $\GG $ for the filtration generated by the first two co-ordinates $Z^1 $ and $Z^2 $ of the process, and $\FF^X $ for the sub-filtration generated by $X $ interpreted as first co-ordinate of $Z.$

The new process $Z$ is a Markov process with respect to its filtration $\FF .$  For a proof of this result, the interested reader is referred to Theorem 2.7 of L\"ocherbach and Loukianova (2008). In general, $Z$ will no longer be strong Markov. But for any $n \geq 0,$ by construction, the strong Markov property holds with respect to $T_n.$ Thus for any $f,g : E \times [0,1] \times E \to \RR $ measurable and bounded, for any $s > 0$ fixed, for any initial measure $\pi $ on $(E, {\cal E}),$
$$ \E_\pi(g(Z_{T_n}) f(Z_{T_n +s })) = \E_\pi( g(Z_{T_n}) \E_{Z_{T_n}} ( f(Z_s)) ) .$$
Finally, an important point is that by construction, 
$$ {\cal L} ( (Z_t^1)_t | \P_x ) = {\cal L} (( X_t)_t | P_x )$$
for any $x \in E,$ thus the first co-ordinate of the process $Z$ is indeed a copy of the original Markov process $X,$
when disregarding the additional colours $ (Z^2, Z^3).$ 

However, adding the colours $ (Z^2, Z^3)$ allows to introduce regeneration times for the process $Z $ (not for $X$ itself).
More precisely, write
$$ A := C \times [0, \alpha ] \times E $$
and put
\begin{multline}\label{eq:regtimes}
 S_0 := 0, \; R_0 := 0 , S_{n+1} := \inf \{ T_m > R_n:  Z_{T_m} \in A  \},\\ R_{n+1} := \inf\{ T_m : T_m > S_{n+1} \}  .
\end{multline}
The sequence of $\FF -$stopping times $R_n$ generalises the notion of life-cycle
decomposition in the following sense.

\begin{prop}\label{iid} [Proposition 2.6 and 2.13 of L\"ocherbach and Loukianova (2008)]\\
a) Under $\P_x ,$ the sequence of jump times $(T_n)_n$ is independent of the first co-ordinate process $(Z_t^1)_t$ and $ (T_n - T_{n-1})_n$ are i.i.d.~$exp(1)-$variables.\\
b) At regeneration times, we start from a fixed initial distribution which does not depend on the past: $Z_{R_n}   \sim \nu(dx) U(du) Q((x,u), dx') $ for all $n \geq 1 .$\\
c) At regeneration times, we start afresh and have independence after a waiting time: $Z_{R_n + \cdot}$ is independent of ${\cal F}_{S_{n}-} $ for all $ n \geq 1.$\\
d) The sequence of $(Z_{R_n})_{n\geq 1} $ is i.i.d.
\end{prop}

Since the original process $X$ -- under Assumption \ref{ass:driftcond} -- is Harris with invariant measure $\mu,$ the new process $Z$ will be Harris, too. We shall write $\Pi$ for its 
invariant probability measure. $\Pi$ can be written in terms of an occupation time formula which is a consequence of 
Chacon-Ornstein's ratio limit theorem. In order to state this theorem, let us recall
that an additive functional of the process $Z$ is a $\bar{\RR}_+ -$valued, $\FF-$adapted process $A = (A_t)_{t \geq 0} $ such that
\begin{enumerate}
\item
Almost surely, the process is non-decreasing, right-continuous, having $A_0 = 0. $
\item
For any $s, t \geq 0, $ $A_{s+ t } = A_t + A_s \circ \theta_t $ almost surely. Here, $\theta $ denotes the shift operator.  
\end{enumerate}
The additive functional is called integrable if $ \E_\Pi ( A_1) < + \infty .$ 
Examples for integrable additive functionals are $ A_t = \int_0^t f( Z_s) ds , $ where $f$ is a positive measurable
function, integrable with respect to the invariant measure $\Pi .$ 

\begin{prop}[Chacon-Ornstein's ratio limit theorem]\label{prop:harris}
Let $A_t, B_t $ be any positive additive functionals of $Z$ such that $ \E_\Pi ( B_1) > 0 .$ Then 
$$  \frac{A_t}{B_t} \to \frac{ \E_\Pi ( A_1) }{ \E_\Pi ( B_1)} \quad\P_x-\mbox{ almost surely, as } t \to \infty ,$$
for any $x \in E.$ Moreover, $Z$ is recurrent in the sense of Harris and its unique invariant probability measure $\Pi $ is given by 
\begin{equation}\label{eq:fixedmu}
 \Pi (f) = \ell \; \E_\pi  \int_{R_1}^{R_2} f(Z_s) ds ,
\end{equation}
where $ \ell =  \E (R_2 - R_1)^{-1} > 0.$
\end{prop} 

\begin{proof}
The proof follows easily from the regeneration property with respect to the regeneration times $R_n.$ 
\end{proof}

The invariant measure $\mu$ of the original process $X$ is the projection onto the first co-ordinate of $\Pi. $ 
From this we deduce that the 
invariant probability measure $\mu $ of the original process $X$ must be given by 
\begin{equation}\label{eq:fixedmu2}
 \mu (f) = \ell \; \E_\pi  \int_{R_1}^{R_2} f(X_s) ds ,
\end{equation}
where we recall that $ \ell =  \E (R_2 - R_1)^{-1} > 0.$ In the above formula we interpret $X$ as first co-ordinate of $Z$, under 
$\P_\pi$ \footnote{Actually, we should write $ \E_\pi  \int_{R_1}^{R_2} f(Z^1_s) ds $ -- but if not otherwise indicated, this
identification will always be implicitly assumed.}. $R_2 - R_1 $ is the length of one regeneration period. 
Under assumption (\ref{ass:driftcond}), the process is positive recurrent and hence the expected
length $\ell$ of one regeneration period is finite. 

We now turn to the main issue of this article which is the control of the speed of convergence in the ergodic theorem.  
As a consequence of the above considerations, we can write
\begin{multline}\label{eq:speed}
 P_x \left( \left| \frac1t \int_0^t f(X_s) ds - \mu (f)\right| > \delta \right) \\
= \P_x  \left( \left| \frac1t \int_0^t f (Z^1_s) ds - \ell \; \E_\pi  \int_{R_1}^{R_2} f(Z^1_s) ds\right| > \delta \right) ,
\end{multline} 
where we recall that $\P_x $ denotes the measure related to $Z$ under which $Z_0 \sim \delta_x \otimes U(du) \otimes Q((x,u), dy) .$ 
The more moments of the regeneration period $ R_2 - R_1 $ exist, the more
the process is recurrent and the more the convergence in (\ref{eq:speed}) is fast.

We first give estimates on the polynomial moments 
$$ \E_x \int_0^{R_1} r(s) ds ,$$
depending on the starting point $x.$ Integrating this against $\nu (dx) $ gives then a control on the corresponding moment of the regeneration period. This integration does not pose any problems because the support of the measure $\nu$ is the compact set $C.$ 
Since our regeneration times are built based on the resolvent chain, the main technical ingredient that allows such a control will be 
the estimate (\ref{eq:modmomentres}) rather than (\ref{eq:modmoment2}). 

\section{Polynomial moments of regeneration times}\label{sec:polmoment}
The aim of this section is to show that the results of Douc et al.~(2009) can be translated immediately
into a control of moments of regeneration times. This yields somehow a link between the two 
different approaches ``Drift conditions'' versus ``Nummelin''.  Recall the definition of $r(s) = r_\Phi (s) $ in (\ref{eq:ratefunction}).
\begin{theo}\label{theo:polymoment}
Grant assumptions \ref{regularitypt} and \ref{ass:driftcond} with a function $\Phi (v) = c v^\alpha ,$ where $0 \le \alpha < 1.$ Then there exist constants $c_1$ and $c_2,$ such that 
$$
 \E_x \int_0^{ R_1} r(s) ds \le c_1  V(x)  + c_2  .
$$
\end{theo}

\begin{rem}
For $\Phi (v) = c v^\alpha ,$ it can be easily shown that there exists a constant $c$ such that $r(s) =  r_\Phi (s)  \geq c \, s^{\frac{\alpha}{1 - \alpha }} .$ 
Hence the above theorem implies the control of polynomial moments of the regeneration time, i.e.  
\begin{equation}\label{eq:wichtig}
 \E_x R_1^{ \frac{1}{1 - \alpha }} \le \tilde c_1 V(x) + \tilde c_2 .
\end{equation}  
\end{rem}

\begin{proof}
Recall the definition of the regeneration times in (\ref{eq:regtimes}). Let 
$$ \tilde S_1 := \inf \{ T_n : Z_{T_n}^1 \in C \} , \; \tilde S_{n+1} := 
\inf \{ T_k > \tilde S_n :  Z_{T_k}^1 \in C \} .$$
Obviously, $ R_1 \geq \tilde S_1  .$  \\

1. In the following, $c$ will denote a constant that might change from line to line. We first show how to control 
$$ \E_x \int_0^{\tilde S_1} r(s) ds .$$
In a first step we show that 
\begin{equation}\label{eq:firstly}
  \E_x \int_0^{\tilde S_1} r(s) ds = \E_x \int_0^\infty e^{- \int_0^t 1_C ( Z_s^1 ) ds } r(t) dt = E_x \int_0^\infty e^{- \int_0^t 1_C ( X_s ) ds } r(t) dt .
\end{equation}
This can be seen as follows. First, in order to obtain the law of $\tilde S_1,$ we evaluate for any $a > 0,$
\begin{eqnarray*}
\P_x (\tilde S_1 > a) &=& \sum_{ n \geq 1} \P_x ( \tilde S_1 = T_n, T_n > a ) \\
&=& \sum_{ n \geq 1} \P_x ( Z^1_{T_1 } \in C^c , \ldots , Z^1_{T_{n-1}} \in C^c , Z^1_{T_n} \in C, T_n > a ) \\
&=& \sum_{ n \geq 1} P_x ( X_{T_1 } \in C^c , \ldots , X_{T_{n-1}} \in C^c , X_{T_n} \in C, T_n > a ) \\
&=& E_x \left( \sum_{ n \geq 1} ( 1 - 1_C ( X_{T_1 })) \cdots ( 1 - 1_C ( X_{T_{n-1} })) f( X_{T_n}, T_n) \right) , 
\end{eqnarray*}
where $f(t,x) = 1_{ t > a} 1_C (x) .$

Now, we make use of the following very useful formula which is taken from
H\"opfner and L\"ocherbach (2003), (5.29), page 59. 
\begin{eqnarray*} &&E_x \left( \sum_{ n \geq 1} ( 1 - 1_C ( X_{T_1 })) \cdots ( 1 - 1_C ( X_{T_{n-1} })) f( X_{T_n}, T_n) \right) \\
&&= E_x \left( \int_0^\infty f(t, X_t) e^{ - \int_0^t 1_C (X_s ) ds } dt \right) \\
&&= E_x \left( \int_a^\infty 1_C (X_t) e^{ - \int_0^t 1_C (X_s ) ds } dt \right)  .
\end{eqnarray*}
Hence we obtain
$$ \P_x (\tilde S_1 > a) =E_x \left( \int_a^\infty 1_C (X_t) e^{ - \int_0^t 1_C (X_s ) ds } dt \right) = E_x \left( e^{ - \int_0^a 1_C (X_s ) ds } \right) . $$
Writing finally that 
$$ \E_x \int_0^{\tilde S_1} r(s) ds = \E_x \int_0^\infty 1_{ s < \tilde S_1} r(s) ds =
\int_0^\infty r(s) \P_x (\tilde S_1 > s ) ds  , $$ 
we get (\ref{eq:firstly}). 
No we apply once more formula (5.29) of H\"opfner and L\"ocherbach (2003) and obtain
\begin{equation}\label{eq:modmoment4}
 E_x \int_0^\infty e^{- \int_0^t 1_C ( X_s ) ds } r(t) dt  = E_x \left( \sum_{n=1}^\infty (1 - 1_C ( \bar X_1)) \cdots (1 - 1_C ( \bar X_{n-1})) r( T_n) \right) ,
 \end{equation}
where we recall that $\bar X_n = X_{T_n}$ is the process observed at the $n-$th jump time of an independent rate one Poisson process. 
The expression at the right hand side of (\ref{eq:modmoment4}) is almost a modulated moment for the resolvent chain, except that we have to replace
$r (T_n) $ by $r(n) .$ This is not difficult since for $n$ large we can use the law of large numbers. Since $r$ is increasing we can write 
\begin{eqnarray}\label{eq:modmoment5}
&&E_x \left(  (1 - 1_C ( \bar X_1)) \cdots (1 - 1_C ( \bar X_{n-1})) r( T_n) \right) \\
&&\le E_x \left(  (1 - 1_C ( \bar X_1)) \cdots (1 - 1_C ( \bar X_{n-1})) r( 2n) \right)
\nonumber \\&&
\quad \quad \quad \quad \quad + E_x \left(  (1 - 1_C ( \bar X_1)) \cdots (1 - 1_C ( \bar X_{n-1})) 1_{ T_n > 2 n } r( T_n) \right) .
\end{eqnarray}
Let us start with the control of the first term in this decomposition. Recall that $\bar \tau_C = \inf \{ n \geq 1  : \bar X_n \in C \} .$ Now, 
using that $ r(2n) \le c r(n),$ which follows from $r(t+s) \le c ( r(t) + r(s))$ by (\ref{eq:subadditive}), 
\begin{eqnarray}\label{eq:modmoment6}
&&\E_x \left( \sum_{n=1}^\infty (1 - 1_C ( \bar X_1)) \cdots (1 - 1_C ( \bar X_{n-1})) r( 2 n) \right) =  
\E_x \left( \sum_{n=1}^{\bar \tau_C} r(2n) \right)  \nonumber \\
&& \le  c \E_x \left( \sum_{n=1}^{\bar \tau_C} r(n) \right)\le c \E_x \left( \sum_{n=1}^{\bar \tau_C- 1} r(n) \right) + c \E_x r ( \bar \tau_C)  . 
\end{eqnarray}
Let $R (k) = \sum_{ j = 0}^{ k - 1} r (j).$ Since $r$ is polynomial,  $ \lim_{ k \to \infty } r(k)/R(k) = 0.$ Hence there exists
a constant $c$ such that for all $ k \geq 1, $ $ r(k) \le R(k) + c .$ As a consequence,
$$ \E_x r ( \bar \tau_C) \le c + \E_x \left( \sum_{n=0}^{\bar \tau_C- 1} r(n) \right) .$$
Using (\ref{eq:modmomentres}), we can thus conclude that 
$$  \E_x \left( \sum_{n=1}^\infty (1 - 1_C ( \bar X_1)) \cdots (1 - 1_C ( \bar X_{n-1})) r( 2 n) \right) \le  c_1 V (x) + c_2 
 .$$ 
Now we turn to the second expression in (\ref{eq:modmoment5}) above: For any $ 1 \le p, q $ such that $\frac1p + \frac1q = 1,$ 
\begin{eqnarray}\label{eq:important1}
&&\E_x \left(  (1 - 1_C ( \bar X_1)) \cdots (1 - 1_C ( \bar X_{n-1})) 1_{ T_n > 2 n } r( T_n) \right)  \nonumber \\
&& \le \left[\E_x r^p ( T_n) \right]^{1/p} \cdot \left[\P_x (T_n > 2 n )\right]^{1/q}\nonumber \\
&& \le \left[\E_x r^p ( T_n) \right]^{1/p} \cdot e^{-C n} 
\end{eqnarray}
for some suitable constant $C.$ But $r^p (\cdot) $ is polynomial and $T_n$ the sum of $n$ independent $\exp (1) $ variables, hence
$ \sup_x \E_x r^p ( T_n) \le P(n) ,$ where $P(.) $ is a polynomial in $n.$ As a consequence,
$$ \sum_{n \geq 1}  \sup_{x } \E_x \left(  (1 - 1_C ( \bar X_1)) \cdots (1 - 1_C ( \bar X_{n-1})) 1_{ T_n > 2 n } r( T_n) \right) = C_2 < \infty .$$
Putting together (\ref{eq:firstly}), (\ref{eq:modmoment4})--(\ref{eq:important1}), we thus get that 
\begin{equation}\label{eq:firstresult}
 \E_x \int_0^{\tilde S_1} r(s) ds \le c_1 V(x)  + c_2 .
\end{equation} 
This will be the main contribution to the control of $\E_x \int_0^{R_1} r(s) ds .$ In the sequel, we shall also use that (\ref{eq:firstresult}) implies in
particular   
\begin{equation}\label{eq:inpart}
\sup_{ x \in C}  \E_x \int_0^{\tilde S_1} r(s) ds < + \infty ,
\end{equation}
since $C$ is compact.\\
2. Recall the definition of $S_1$ in (\ref{eq:regtimes}). We now show how to use the control of $\tilde S_1$ in order to obtain a control of $S_1 .$ We have, since $ r(t+s ) \le r(s) r(t),$
\begin{eqnarray}\label{eq:zweiterterm}
\E_x \int_0^{S_1} r(s) ds & = & \E_x \int_0^{\tilde S_1} r(s) ds + \sum_{n \geq 1} \E_x \left( \int_{\tilde S_n}^{\tilde S_{n+1}} r (s) ds 1_{ \tilde S_n < S_1} \right)\nonumber \\
& = & \E_x \int_0^{\tilde S_1} r(s) ds + \sum_{n \geq 1} \E_x \left(  \int_{0}^{\tilde S_{n+1}- \tilde S_n} r (\tilde S_n + s) ds 1_{ \tilde S_n < S_1} \right)\nonumber \\
&\le & \E_x \int_0^{\tilde S_1} r(s) ds \nonumber\\
&& +  \sum_{n \geq 1} \E_x \left(\left[  \int_{0}^{\tilde S_{n+1}- \tilde S_n}  \! \! \! \! \! r (s) ds\right]  r( \tilde S_n)  1_{ \tilde S_n < S_1} \right).
\end{eqnarray}
The first term in this expression can be controlled using (\ref{eq:firstresult}). We study the second term in the above expression 
$$ \E_x \left( r( \tilde S_n)  1_{ \tilde S_n < S_1} \int_{0}^{\tilde S_{n+1}- \tilde S_n}  \! \! \! \! \! r (s) ds 
  \right) .$$ 
We know that  
$ \P_x ( \tilde S_n < S_1 ) = (1 - \alpha )^n $ (see for example the proof of Proposition 2.16 in L\"ocherbach and Loukianova (2008)). A first idea would be to use Markov's property with respect to $\tilde S_n : $ 
$$  \E_x \left(  r( \tilde S_n)  1_{ \tilde S_n < S_1} \int_{0}^{\tilde S_{n+1}- \tilde S_n}  \! \! \! \! \! r (s) ds 
   \right)= \E_x \left(r( \tilde S_n)  1_{ \tilde S_n < S_1}  \E_{ Z_{\tilde S_n}} \int_0^{\tilde S_1} r(s) ds  \right) .$$
But unfortunately it is {\bf not true} that 
$$ \E_{ Z_{\tilde S_n}} \int_0^{\tilde S_1} r(s) ds  \le \sup_{x \in C} \E_x \int_0^{\tilde S_1} r(s) ds ,$$
we only have that on $\{ \tilde S_n < S_1 \} ,$
$$ \E_{ Z_{\tilde S_n}} \int_0^{\tilde S_1} r(s) ds  \le \sup_{x \in C, u > \alpha , z \in E } \E_{(x, u, z) }  \int_0^{\tilde S_1} r(s) ds ,$$
and this can not be directly controlled using (\ref{eq:firstresult}). 

Hence, we must be more careful. We use that $ r(\tilde S_n) 1_{\{ \tilde S_n < S_1 \}}  $ is measurable with respect to ${\cal G}_{\tilde S_n}$ where we recall that $({\cal G}_t)_t$ is the filtration generated by the first two co-ordinates $Z^1 $ and $Z^2 $ of $Z.$ Hence we will condition on ${\cal G}_{\tilde S_n}.$ Note that by construction of $Z,$ this means that we condition on the whole history of the whole process, i.e. the three co-ordinates, up to the last jump time $\sup \{ T_k : T_k < \tilde S_n\} $ strictly before $\tilde S_n,$ and on the history of $Z^1 $ and $Z^2 $ up to time $\tilde S_n.$  In other words, conditioning on ${\cal G}_{\tilde S_n},$ we know $Z^1_{\tilde S_n} $ and $Z^2_{\tilde S_n} ,$ but $Z^3_{\tilde S_n} $ has still to be chosen. Moreover, on $ \{ \tilde S_n < S_1 \},$ $Z^2_{\tilde S_n} > \alpha ,$ and hence the second line in the definition of the kernel $Q((x,u), dx') $ of (\ref{Q}) has to be applied. 

Write $\nu (x)$ for the density of $ \nu (dx)$ with respect to the dominating measure $\Lambda (dx) $ of assumption \ref{regularitypt}. Then,
\begin{eqnarray}\label{eq:wichtig2}
&&\E_x \left( r (\tilde S_n )  1_{ \tilde S_n < S_1} \int_{0}^{\tilde S_{n+1}- \tilde S_n} r (s) ds  \right) \nonumber \\
&& = 
\E_x \left(r (\tilde S_n ) 1_{ \tilde S_n < S_1}  \int_{E} \frac{1}{1 - \alpha} \left[ u^1 ( Z^1_{\tilde S_n}, x') - \alpha \nu (x') \right] \Lambda (dx')  \right.\nonumber \\
&&\quad \quad \quad \quad \quad \quad  \quad \quad \quad \quad \quad \quad \quad \quad \quad \left. 
\E_{ (Z^1_{\tilde S_n}, Z^2_{\tilde S_n}, x')} \int_0^{\tilde S_1} r(s) ds 
\right) \nonumber  \\
&& \le \frac{1}{1 - \alpha} \E_x \left(r (\tilde S_n ) 1_{ \tilde S_n < S_1}  \int_{E}  u^1 ( Z^1_{\tilde S_n}, x')   \Lambda (dx') 
\right. \nonumber \\
&&\quad \quad \quad \quad \quad \quad  \quad \quad \quad \quad \quad \quad \quad \quad \quad\left. 
\E_{ (Z^1_{\tilde S_n}, Z^2_{\tilde S_n}, x')} \int_0^{\tilde S_1} r(s) ds 
\right) .
\end{eqnarray}
But for any $x,u,$
\begin{multline}\label{eq:wichtig3}
 \int_{E}  u^1 ( x , x')   \Lambda (dx') \E_{(x,u,x')} \int_0^{\tilde{S_1}} r(s) ds \\
= \int_0^1 du \int_E Q((x,u), dx') \E_{(x,u,x')} \int_0^{\tilde{S_1}} r(s) ds    ,
\end{multline} 
since $ \E_{(x,u,x')} \int_0^{\tilde{S_1}} r(s) ds $ does not depend on $u.$ Moreover, 
$$ \int_0^1 du \int_E Q((x,u), dx') \E_{(x,u,x')} \int_0^{\tilde{S_1}} r(s) ds   = \E_x \int_0^{\tilde{S_1}} r(s) ds .$$
Hence, since $ Z^1_{\tilde S_n} \in C, $  
\begin{eqnarray}\label{eq:auchauch}
&&\E_x \left( r (\tilde S_n )  1_{ \tilde S_n < S_1} \int_{0}^{\tilde S_{n+1}- \tilde S_n} r (s) ds  \right) \\
&& \le  \frac{1}{1 - \alpha} \E_x \left(r (\tilde S_n ) 1_{ \tilde S_n < S_1}  
  \left( \sup_{ x \in C} \E_x \int_0^{\tilde S_1} r(s) ds  \right)  \right) \nonumber \\
&& \le  \frac{c}{1 - \alpha}   \E_x \left(r (\tilde S_n ) 1_{ \tilde S_n < S_1}   \right). 
\end{eqnarray}  

Hence we must control $ \E_x (  1_{ \tilde S_n < S_1} r ({\tilde S_n} ) ) .$
We write $\tilde S_n  = \tilde S_1 + (\tilde S_n- \tilde S_1)  $
and use once more the sub-multiplicativity of $r.$ 
We obtain 
\begin{equation}
\E_x \left( r(\tilde S_n) 1_{ \tilde S_n < S_1}\right) \le  \E_x \left(   r(\tilde S_1)1_{ \tilde S_1 < S_1} r(\tilde S_n - \tilde S_1) 1_{ \tilde S_n < S_1}\right)   .
\end{equation}
Here, we have cut $ \tilde S_n = \tilde S_1 + ( \tilde S_n - \tilde S_1) $ into two pieces in order
to get a last term which does not depend on the starting point. The same arguments as above in (\ref{eq:wichtig2})
and (\ref{eq:wichtig3}) yield, when conditioning on ${\cal G}_{\tilde S_1} ,$ the following. 
\begin{eqnarray}\label{eq:first1}
&& \E_x \left( r(\tilde S_n) 1_{ \tilde S_n < S_1}\right)  \le \E_x \left(   r(\tilde S_1)1_{ \tilde S_1 < S_1} r(\tilde S_n - \tilde S_1) 1_{ \tilde S_n < S_1}\right) \nonumber \\
&&\le    \E_x \!\!\left(\!\! r( \tilde S_1) 1_{ \tilde S_1 < S_1} \frac{1}{1 - \alpha}\!\! \int_E \!\!u^1 ( Z_{\tilde S_1}^1, x') \Lambda (dx') \E_{( Z_{\tilde S_1}^1, Z_{\tilde S_1}^2, x')} [ r( \tilde S_{n-1}) 1_{ \tilde S_{n-1} < S_1 } ]  \right) \nonumber \\
&&\le  \frac{1}{1 - \alpha } \E_x \left(   r( \tilde S_1) 1_{ \tilde S_1 < S_1} \, \E_{ Z_{\tilde S_1}^1}  [ r( \tilde S_{n-1}) 1_{ \tilde S_{n-1} < S_1 } ]  \right) \nonumber\\
&& \le  \frac{1}{1 - \alpha } \sup_{ y \in C} \E_y  \left( r( \tilde S_{n-1}) 1_{ \tilde S_{n-1} < S_1 } \right) \;   \E_x \left(   r( \tilde S_1) 1_{ \tilde S_1 < S_1} \right) .
\end{eqnarray}
Concerning the last term in the above expression, we use that $r (t) \le  \int_0^t r(s) ds + c $ for some constant $c$ and obtain
\begin{eqnarray}\label{eq:auchauchauch}
\E_x \left( r( \tilde S_1) 1_{ \tilde S_1 < S_1} \right) & \le &c +  \E_x \left( \int_0^{\tilde S_1} r(s) ds  \right) \nonumber \\
& \le &  c +  c_1 V(x) + c_2  =  c_1 V(x) + \tilde c_2  ,
\end{eqnarray}
using (\ref{eq:firstresult}).

Concerning the first term in (\ref{eq:first1}), for $p, q \geq 1 $ such that $\frac1p + \frac1q = 1,$ 
we obtain
\begin{eqnarray}\label{eq:horrible}
\sup_{y \in C}  \E_y \left( r(\tilde S_{n-1}) 1_{ \tilde S_{n-1} < S_1 } \right)   &\le& \sup_{y \in C} \left( \E_y  r^p(\tilde S_{n-1} )\right)^{1/p} \P_y ( \tilde S_{n-1} < S_1)^{1/q}    \nonumber  \\
& \le   &  (1- \alpha)^{(n-1)/q} \left(  \sup_{y \in C} \E_y  r^p(\tilde S_{n-1} ))\right)^{1/p}  .
\end{eqnarray}
We have to control this last expression. We claim the following: There exists a $\kappa > 0 $ and a constant $c$ such that 
for $p > 1 $ sufficiently small, 
\begin{equation}\label{eq:claim}
\left(  \sup_{y \in C} \E_y  r^p(\tilde S_{n-1} ))\right)^{1/p} \le c n^\kappa .
\end{equation}
Once (\ref{eq:claim}) is proven, we obtain, using (\ref{eq:zweiterterm}), (\ref{eq:auchauch}), (\ref{eq:first1}), 
(\ref{eq:auchauchauch}), (\ref{eq:horrible}) and (\ref{eq:claim}) the following: 
\begin{eqnarray}\label{eq:uffuff}
&&\E_x \int_0^{S_1} r(s) ds \nonumber \\
&& \le  (c_1 V(x) + c_2 ) + \frac{c}{(1 - \alpha)^2} ( c_1 V(x) + \tilde c_2) \sum_{n \geq 1} (1- \alpha)^{(n-1)/q} n^\kappa \nonumber  \\
&& = \bar{c}_1 V(x) + \bar{ c}_2 .
\end{eqnarray}

It remains to show (\ref{eq:claim}): 
By our assumptions, $r$ is polynomial and $ r(x) \sim C x^{\frac{\alpha}{1 - \alpha}} $ as $x \to \infty ,$ hence $r^p (x) \le c x^{\kappa p } ,$ where $ \kappa = \alpha / ( 1 - \alpha) .$  We now fix the choice of $p$ and $q$ in (\ref{eq:horrible}). We choose 
$$p \in \; ] \frac{1}{\kappa}, 1 + \frac{1}{\kappa}[ \; = \; 
] \frac{1 - \alpha}{\alpha}, \frac{1}{\alpha }[ .$$ 
Then $ \kappa p \geq 1 ,$ and we use Jensen's inequality to obtain
\begin{equation}\label{eq:important2}
  r^p(\tilde S_{n-1} ) \le c \tilde S_{n-1} ^{\kappa p } 
  \le (n- 1)^{p\kappa - 1} \left( \tilde S_1^{\kappa p} + \ldots + (\tilde S_{n-1} - \tilde S_{n-2})^{\kappa p} \right) .
\end{equation}

Now since $p < 1 + \frac{1}{\kappa} = \frac{1}{\alpha} ,$ we have $t^{\kappa p}  \le c \int_0^t r(s) ds $ for some constant $c. $ Then for any of the above terms ($k \geq 2$), by (\ref{eq:inpart}),
$$ \sup_{ y \in C} \E_y (\tilde S_k - \tilde S_{k-1})^{\kappa p} \le c \sup_{ y \in C} \E_y \int_0^{\tilde S_1} r(s) ds < \infty .$$
As a consequence, coming back to (\ref{eq:important2}),
$$ \sup_{y \in C} \E_y  r^p ( \tilde S_{n-1} )  \le c ( n-1) ^{p \kappa} \sup_{ y \in C} \E_x \int_0^{\tilde S_1} r(s) ds = \tilde c ( n-1) ^{p \kappa}  ,$$
and this yields (\ref{eq:claim}).

3. Finally we proceed to the control of $R_1 .$ Clearly, 
$$ \E_x \int_0^{R_1} r(s) ds \le \E_x \int_0^{S_1} r(s) ds + \E_x\left[  r(S_1)  \int_0^{R_1 - S_1} r(s) ds \right] .$$
We have to control the last term above. We condition on ${\cal G}_{S_1},$ notice that $ Z^2_{S_1} \le \alpha  $ and use step 1. of the construction of $Z,$ hence
\begin{eqnarray*}
&& \E_x\left[  r(S_1)  \int_0^{R_1 - S_1} r(s) ds \right] \\
&&= \E_x\left[  r(S_1) \left( \int_E  \nu (x')\Lambda (dx')  \int_0^\infty e^{-t} \frac{p_t (Z_{S_1}^1, x')}{u^1 (Z_{S_1}^1, x') } dt  \int_0^{t} r(s) ds  \right) \right] .
\end{eqnarray*}
But by (\ref{minoration}), $ \nu (x') \le \frac{1}{\alpha}   u^1 (Z_{S_1}^1, x'),$ since $ Z_{S_n}^1 \in C,$ thus 
\begin{eqnarray*} && \E_x\left[  r(S_1)  \int_0^{R_1 - S_1} r(s) ds \right] \\
&&\le \frac{1}{\alpha}  \E_x\left[  r(S_1) \left( \int_E \Lambda (dx') \int_0^\infty e^{-t} p_t (Z_{S_1}^1, x')dt  \int_0^{t} r(s) ds  \right) \right]  \\
&&=\frac{1}{\alpha}  \E_x\left[  r(S_1) \int_0^\infty e^{-t}dt  \left(  \int_E p_t (Z_{S_1}^1, x') \Lambda (dx') \right)  [\int_0^{t} r(s) ds] \right] \\
&&=\frac{1}{\alpha}  \E_x\left[  r(S_1) \int_0^\infty e^{-t} dt \int_0^{t} r(s) ds  \right] \\
&&= \frac{c}{\alpha} \E_x (r(S_1)),
\end{eqnarray*}
since $\int_0^\infty e^{-t}\int_0^{t} r(s) ds dt < \infty .$ Finally, $r(t) \le \int_0^t r(s) ds + c $ gives  
$$ \E_x (r(S_1)) \le \E_x \int_0^{S_1} r(s) ds + c ,$$ which is controlled due to (\ref{eq:uffuff}). 
This concludes the proof. 
\end{proof}

\begin{rem}
The fact that the rate function is polynomial was crucial at two points in the above proof: in equations (\ref{eq:important1}) and (\ref{eq:claim}). The general sub-geometrical case could probably be handled by paying in particular attention to the constants that arrive in expressions like $ \E_x r^p (T_n) \le [\E_x r^p (T_1)]^n .$ 
\end{rem}

\section{Polynomial deviation inequality}\label{lastsection}

We impose Assumption \ref{ass:driftcond} with a function $\Phi (v) = c v^\alpha, $ where $ 0 \le \alpha < 1.$ 
As a consequence, we obtain a control for polynomial moments $\E_x R_1^p $ of the regeneration time for all $ p \le 1/ (1- \alpha),$
due to (\ref{eq:wichtig}). Since $V$ is continuous and since the measure $\nu $ of (\ref{minoration}) which is used in order
to construct the regeneration periods is of compact support, also $\E_\nu R_1^p$ is finite for all $ p \le 1/ (1- \alpha).$ 

In order to derive the deviation inequality we first state a deviation inequality for the counting process associated to the life cycle decomposition 
$$  N_t = \sup \{ n : R_n \le t \}=\sum_{n=1}^{\infty} 1_{\{ R_n \le t \}} , \; N_0 = 0 . $$ 
We have almost surely, as $t \to \infty, $ $N_t / t \to \E_\Pi N_1 = \ell ,$ where we recall that
$$ \ell = (\E_\nu R_1)^{- 1} = ( \E (R_2 - R_1))^{-1} ,$$
see Proposition \ref{prop:harris} and equation (\ref{eq:fixedmu}).

The deviation inequality for the counting process associated to the life cycle decomposition is the following. 

\begin{theo}\label{theo:deviationnt}
Grant Assumptions \ref{regularitypt} and \ref{ass:driftcond} with $\Phi (v) = c v^\alpha ,$ $0 \le \alpha < 1.$
Let $x \in E$ be any starting point and $0< \ge <  1. $  
Then for any $1 < p \le  1/ (1 - \alpha ) $ there exists a positive constant $C(\ell,p ,\nu)$ such that 
the following inequality holds:
\begin {equation}
\P_{x}\left(|\frac{N_t}{t}-\ell|>\ell\ge\right )\leq 
C(\ell,p, \nu) V(x) \; 
\frac1{\ge^p \wedge \ge^{ 2 (p-1)}}  \; \frac1{t^{ p-1}} .
\end{equation}
Here  $C(\ell,p,\nu)$ is given by 
$$C(\ell,p,\nu) =
\left\{ 
\begin{array} {ll}
C(p) \left[ (1 + (1/\ell )^{p- 1})+  (m_p  + \sigma^{2(p-1)} ) \right]    \, \left( \ell^{p-1} \vee \ell \right) & \mbox{ if } p \geq 2\\
C(p) \left[ ( 1 + (1/\ell )^{p-1}) + m_p \ell C_p^p \right]   &   \mbox{ if } p \in ] 1, 2 [ 
\end{array}
\right\} ,
$$
where $C(p)$ is a constant depending only on $p,$ $C_p$ is the constant of the Burkholder-Davis-Gundy inequality, $m_p = \E ( R_2 - R_1 - \frac{1}{\ell} )^p $ and $\sigma^2 = Var ( R_2 - R_1 ),$ in the case $p \geq 2 .$ 
\end{theo}

\begin{proof}
The proof is basically the same as in L\"ocherbach, Loukianova 
and Loukianov (2011), proof of Theorem 3.1. Put in contrary to there we use the Fuk-Nagaev inequality given in the appendix (Theorem \ref{theo:fuknagaev}) in the case $ p \geq 2 .$ We decompose:
\begin{equation}\label{Ntfirst}
\P_{x}\left(|N_t/t-\ell|>\ell\ge\right )
\leq \P_{x}\left(N_t/t>\ell(1+\ge)\right )+\P_{x}\left(N_t/t<\ell(1-\ge)\right ) .
\end{equation}

Put for $k\geq 1,$
$\bar{\eta}_k=-1(R_{k+1}-R_{k}-1/\ell).$
For the first term of (\ref{Ntfirst}), we have 
\begin{equation}\label{Ntfirst1}
\P_{x}\left(N_t/t>\ell(1+\ge)\right )
\leq \P_{x}\left(R_1- 1/\ell \leq -t\ge/2\right)+
\P_{x}\left(\sum_{k=1}^{[t\ell (1+\ge)]}\bar{\eta}_k \geq t\ge/2\right) .
\end {equation}

In an analogous way,  
\begin{equation}\label{Ntfirst2}
\P_{x}\left(N_t/t<\ell (1-\ge)\right )\leq \P_{x}\left(R_1-\frac1\ell \geq t\ge/2\right)
+\P_{x}\left(\sum_{k=1}^{[t\ell (1-\ge)]-1}\bar{\eta}_k\leq -t\ge/2\right).
\end{equation}
The random variables $\bar{\eta}_k, k \geq 1, $ are identically distributed centred random variables  such that $\E_{x}|\bar{\eta}_k|^p <\infty .$ Moreover, they are two-dependent. Indeed, $\bar \eta_k$ is not independent of ${\cal F}_{R_k} ,$ but only independent of 
${\cal F}_{R_{k-1}} .$ This is due to step 1 of the construction of $Z,$ where the waiting
time for the new jump is chosen depending on the actual value of $Z$ at time $R_k.$ 

If $p \geq 2 ,$ we can apply Theorem \ref{theo:fuknagaev}. Let $M_0=0$ and $M_n=\sum_{k=1}^{n} \bar{\eta}_k.$
Denote $M^*_n=\sup_{k\leq {n}}|M_k| .$  
As a consequence of (\ref{Ntfirst1}) and (\ref{Ntfirst2}) we can write
\begin{equation}\label{Ntsecond}
\P_{x}\left(|N_t/t-\ell |>\ell \ge\right )
\leq  \P_{x}\left(|R_1-1/\ell |\geq t\ge/2\right )
+ \P_{x}\left( M^*_{[t\ell (1+\ge)]}\geq t\ge/2\right)  .
\end{equation}

We use Theorem \ref{theo:fuknagaev} with $ n = [t \ell  ( 1 +\ge)] $ and $\lambda = t \ge / 8 $ and obtain 
\begin {eqnarray*}
&&\P_{x}\left(|N_t/t-\ell |>\ell \ge\right )\\
&& 
\le   \frac{2^{p- 1}\E_{x}|R_1-1/\ell |^{p- 1}}{(t\ge)^{p- 1}}+
 \; C ( p) [m_p + \sigma^{2(p-1)}]    \, \left( \ell^{p-1} \vee \ell \right) \ge^{ - 2 (p-1)} t^{- (p-1)}  \\
& &\leq  \left( 2^{p- 1}\E_{x}|R_1-1/\ell |^{p- 1}+ C ( p) [m_p + \sigma^{2(p-1)}]   \, \left( \ell^{p-1} \vee \ell \right) \right) \frac1{\ge^{ 2 (p-1)} }t^{- (p-1)} ,
\end {eqnarray*}
since $\ge < 1,$ where $m_p =  \E_{x}|\bar{\eta}_1|^p,$ $\sigma^2 = Var (\bar{\eta}_1 ).$ 
Finally we use that 
$$ \E_x |R_1-1/\ell |^{p- 1} \le C(p) [ \E_x R_1^{p-1} + (1/\ell )^{p-1} ] , $$
and that for some constants $c$ and $d,$
$$ \E_x (R_1^{p-1}) \le 1 + \E_x R_1^p \le c V(x) + d  $$
to conclude that, since $V(.) \geq 1,$
\begin{multline*}
 \P_{x}\left(|N_t/t-\ell |>\ell \ge\right ) \le C(p) \, V(x) (  (1 + (1/\ell )^{p- 1}) + \\
 [m_p + \sigma^{2(p-1)}]   \, \left( \ell ^{p-1} \vee \ell \right) ) \frac1{\ge^{ 2 (p-1)} }t^{- (p-1)} .
\end{multline*}
This finishes the proof in the case $ p \geq 2.$ 

In the case $ 1 < p < 2,$ we apply the Burkholder-Davis-Gundy inequality. In order to produce independent random variables, we define
\begin{equation}\label{eq:xi1}
 \eta^{(1)}_k = \left\{
\begin{array}{ll}
\bar{\eta}_k & \mbox{ if $k$ odd} \\
0 & \mbox{ elseif }
\end{array}
\right\} , \; \eta^{(2)}_k = \left\{
\begin{array}{ll}
\bar{\eta}_k & \mbox{ if $k$ even} \\
0 & \mbox{ elseif }
\end{array}
\right\} .
\end{equation}
Let $M^1_0=0$ and $M^1_n=\sum_{k=1}^{n} \eta^{(1)}_k.$ In the same way, $M^2_0 = 0$ and 
$M^2_n = \sum_{ k=1}^n \eta^{(2)}_k .$ 

We also introduce the following two sub-filtrations, 
associated to the sum of odd and the sum of even terms. 
Let
$$ {\cal A}_n^{(1)} := \sigma \{  \eta_k^{(1)} : k \le n , k \mbox{ odd } \} = \sigma \{ M_k^{(1)} , k \le n \} , $$
and 
$$ {\cal A}_n^{(2)} := \sigma \{ \eta_k^{(2)} : k \le n , k \mbox{ even } \}= \sigma \{ M_k^{(2)} , k \le n \} .  $$
Then $(M^1_n)_n$ and $(M^2_n)_n$ are discrete
$ {\cal A}^{(1)}_{n}-$martingales  ($ {\cal A}^{(2)}_{n}-$martingales, respectively). 
Both martingales are $L^p$ martingales such that $[M^{(i)}]_n=\sum_{k=1}^n (\eta^{(i)}_k)^2,$ for $ i = 1, 2.$ Denote $(M^{(i)})^*_n=\sup_{k\leq {n}}|M^{(i)}_k|,$ $i = 1, 2.$ 
In this case, as a consequence of (\ref{Ntfirst1}) and (\ref{Ntfirst2}) we write
\begin{eqnarray}\label{Ntsecondbis}
&&\P_{x}\left(|N_t/t-\ell |>\ell \ge\right )
\leq  \P_{x}\left(|R_1-1/\ell |\geq t\ge/2\right ) \\
&&+ \P_{x}\left( (M^{(1)})^*_{[t\ell (1+\ge)]}\geq t\ge/4\right) + \P_{x}\left( (M^{(2)})^*_{[t\ell (1+\ge)]}\geq t\ge/4\right) . \nonumber
\end{eqnarray}

We use the Burkholder-Davis-Gundy inequality to bound the last terms in (\ref{Ntsecondbis}):  
For all $p>1$ there exists a constant $C_p$ depending only $p$ such that $\|(M^{(i)})_n^*\|_p\leq C_p\|[M^{(i)}]_n^{1/2}\|_p,$  hence
 $\E_{x}((M^{(i)})_n^*)^p\leq C_p^p\E_{x}\left (\sum_{k=1}^n (\eta^{(i)}_k)^2\right)^{p/2}.$

Notice that by definition, the term $\sum_{k=1}^n (\eta^{(1)}_k)^2 $ contains $ [ \frac{n+1}{2}] $ terms whereas 
$\sum_{k=1}^n (\eta^{(2)}_k)^2 $ contains $ [ n/2] $ terms. 
Since $1<p<2,$ the sub-additivity of the function $x \mapsto x^{p/2}$ implies 
\begin{equation}\label{BGD21}
  \left(\sum_{k=1}^{n} (\eta^{(1)}_k)^2\right)^{p/2 }\le 
  \sum_{k=1}^{n}|\bar{\eta}^{(1)}_k|^p ,\quad\mbox{hence}\quad \E_{x} ((M^{(1)})_n^*)^p\leq C_p^p 
  n \E|\bar{\eta}_1|^p.
  \end{equation}
The same kind of bound holds also for the even terms. 

Now we can conclude similarly to L\"ocherbach, Loukianova and Loukianov (2011):   For $1<p<2, $ 
\begin {eqnarray*}
&&\P_{x}\left(|N_t/t-\ell |>\ell \ge\right ) \\
&& \le   \frac{2^{p- 1}\E_{x}|R_1-1/\ell |^{p- 1 }}{(t\ge)^{p- 1 }}+
2 \; 4 ^p C^p_p  \E_{x}|\bar{\eta}_1|^p \,[t\ell (1+\ge)] \frac{1}{(t\ge)^{p}} 
\\
&& \leq \left( 2^{p- 1}\E_{x}|R_1-1/\ell |^{p- 1 }+2^{2p+2 }C^p_p  \E_{x }|\bar{\eta}_1|^p  \, \ell \right) \frac1{\ge^p}\frac{1}{t^{p-1}}\\
&& \le C(p) V(x) \left( ( 1 + (1/\ell )^{p-1}) + m_p \ell C_p^p \right) \frac1{\ge^p}\frac{1}{t^{p-1}} .
\end {eqnarray*}
This concludes the proof. 
\end{proof}

Once the deviation inequality for the counting process $(N_t)_t$ is proven, we obtain on the lines of L\"ocherbach, Loukianova and Loukianov (2011), Theorem 3.2, the following general 
deviation inequality for additive functionals of the original Markov process $X,$ built of bounded functions. 

\begin{theo}\label{mainbounded}
Grant Assumptions \ref{regularitypt} and \ref{ass:driftcond} with $\Phi (v) = c v^\alpha ,$ $0 \le \alpha < 1.$ 
Put $p = 1/ (1 - \alpha ) .$ Let $f \in L^1(\mu)$. Suppose that $\|f\|_\infty <\infty$.
Let $x$ be any initial point and $0< \ge <  \|f\|_\infty  . $ Then for all $t\geq 1$  the following inequality holds:
\begin{eqnarray}\label{polbounded2}
&& P_{x}\left(\left|\frac 1t\int_0^tf(X_s)ds-\mu(f)\right|>\ge\right)\leq 
K(\ell ,p, \nu,X) \, V(x) \, t^{- (p-1) } \times \nonumber \\
&& 
\quad \quad \quad \quad \quad \quad \quad \quad \quad  \quad \quad \quad \times \left\{
\begin{array}{ll}
 \frac 1{\ge^{2(p-1)}}  \| f\|_\infty^{ 2 (p-1)}   & \mbox{ if } p \geq 2 \\
 \frac 1{\ge^p}\|f\|_\infty ^p & \mbox{ if }1< p < 2
\end{array}
\right\}. 
\end{eqnarray}

Here $K(\ell , p, \nu,X)$ is a positive constant, different in the two cases, which depends on $\ell ,p ,\nu$
and on the process $X$ through the life cycle decomposition, but which does not depend on $f$, $t$, $\ge.$
\end{theo}

\begin{rem}
The above result holds for any starting measure $\nu $ such that $ \E_\nu (R_1^p ) $ is finite, so a fortiori for any measure $ \nu $ such that $ V \in L^1 (\nu ) .$ In contrary to most of the existing results in the literature (see e.g. Cattiaux and Guillin (2008)) we do not need to assume absolute continuity of the initial law of the process with respect to the invariant measure $\mu .$ 
\end{rem}

\begin{proof}
First of all, since the law of $X$ starting from a fixed point $x$ is the same as the law of $Z^1 $ starting 
from the initial measure $\P_x,$ we certainly have that 
$$ P_{x}\left(\left|\frac 1t\int_0^tf(X_s)ds-\mu(f)\right|>\ge\right) = \P_x \left(\left|\frac 1t\int_0^tf( Z^1_s)ds-\mu(f)\right|>\ge\right) .$$
Now for $p < 2$ the rest of the proof is exactly the same as the proof of Theorem 3.2 in L\"ocherbach, Loukianova
and Loukianov (2011). The only difference compared to there is that the variables $\xi_n = \int_{R_n}^{R_{n+1}}  (f - \mu (f)) (Z_s^1) ds $
are no longer independent but only $2$-dependent. Hence, the same trick as in the proof of Theorem \ref{theo:deviationnt} applies: one
has to separate even and odd terms. But this does only change the constants in the upper bound. 

For $p \geq 2,$ we use the Fuk-Nagaev inequality again. We start as in the proof of Theorem 3.2 of L\"ocherbach, Loukianova and Loukianov
 (2011). Denote 
$$
\Omega_t=\left\{ \left|\frac{N_t}t-\ell \right|\leq \ell \delta  \right\} , \delta = \ge /\| f\|_\infty < 1 .
$$
Put $\bar f := f - \mu (f) .$ Then 
\begin{eqnarray*}
&& \P_{x}\left(\left|\int_0^tf(Z^1_s)ds-t\mu(f)\right|>t \ge\right)
\nonumber\\
&& \leq\P_{x }\left(\left|\int_0^{R_1}\bar f (Z^1_s)ds\right|>\frac{t\ge }3\right)
+\P_{x}\left(\left|\int_{R_1}^{R_{N_{t}+1}}\bar f (Z^1_s)ds \right|>\frac{t\ge }{3}\ ; \Omega_t\right)\nonumber \\
 &&+\P_{x }\left(\left|\int_t^{R_{N_{t}+1}}\bar f (Z^1 _s)ds\right|>\frac{t\ge }{3}; \Omega_t\right)
+\P_{x}\left(\Omega_t^c\right)\nonumber \\
&&=A+B+C+D . 
\end{eqnarray*}
The terms $A$ and $C$ are handled as in L\"ocherbach, Loukianova and Loukianov (2011). Term $D$ is controlled thanks to Theorem \ref{theo:deviationnt}. So
the main term that has to be controlled is the term $B,$ and we have
$$ B \le \P_{x}\left(\sum_{k=1}^{[t\ell (1+\delta )]} |\xi_k|  \geq t\ge/3 \right), \; \xi_k = \int_{R_k}^{R_{k+1}} \bar f (Z_s^1) ds . $$
Put $\bar{\xi}_k = \frac{1}{\| f\|_\infty } \xi_k, $ then $\bar{\xi}_k = \int_{R_k}^{R_{k+1}} ( g - \mu (g)) (Z_s^1 ) ds ,$ where $g = f / \| f\|_\infty, $ $ \| g\|_\infty = 1 .$ We write
\begin{eqnarray*}
 \P_{x}\left(\sum_{k=1}^{[t\ell (1+\delta )]} |\xi_k|  \geq t\ge/3 \right) &=& \P_{x}\left(\sum_{k=1}^{[t\ell (1+\delta )]} |\bar{\xi}_k|  \geq \frac{t\ge }{3 \|f\|_\infty }  \right)\\
&\le & \P_{x}\left( \sup_{ k \le [t\ell (1+\delta )]} S_k \geq t \delta /3 \right) , \; \delta = \ge / \| f\|_\infty < 1, 
\end{eqnarray*}
$ S_k =  \sum_{i- 1}^k |\bar{\xi}_k| ,$ 
and apply the Fuk-Nagaev inequality of Theorem \ref{theo:fuknagaev} with $n = t \ell ( 1 + \delta) $ and $ \lambda = t \delta / 12 .$ 
This gives the following upper bound 
$$ \P_{x}\left(\sum_{k=1}^{[t\ell (1+\delta )]} |\xi_k|  \geq t\ge/3 \right)  \le C ( p) [ m_p + \sigma^{2(p-1)}]   \, \left( \ell^{p-1} \vee \ell \right) \delta^{ - 2 (p-1)} t^{- (p-1)} ,$$
where 
$$ m_p = \E ( | \bar{\xi}_1|^p ) \le 2^p \, \E ( (R_2- R_1)^p) , \sigma^2 = Var ( \bar{\xi}_1) \le 4 \; \E ( (R_2 - R_1)^2 .$$
Therefore, since $ \delta = \ge / \| f \|_\infty ,$ 
there exists a constant $ K( \ell , p , \nu , X) $ depending only on the process and the life cycle decomposition, but not on the function $f,$ such that 
$$  \P_{x}\left(\sum_{k=1}^{[t\ell (1+\delta )]} |\xi_k|  \geq t\ge/3 \right)  \le K( \ell , p , \nu , X) \| f\|_\infty^{ 2 (p-1)}  \ge^{ - 2 (p-1)} t^{ - (p-1)} . $$
This finishes the proof. 
 
\end{proof}

\section{Examples}
We close our paper with two examples where the above deviation inequalities can be applied. 
\subsection{Multi-dimensional diffusions}
Consider the solution of the following stochastic differential equation in $\RR^d$ 
$$ d X_t = b(X_t) dt + \sigma (X_t) d W_t,$$
where $W_t$ is an $n-$dimensional Brownian motion, $n \geq d,$ such that $ b$ is a 
locally bounded Borel measurable function $\RR^d \to \RR^d$ and $\sigma $ is a bounded
continuous function $\RR^d \to \RR^{d \times n } $ which is uniformly elliptic: Writing $a := \sigma \sigma^*,$ we suppose that there exists $\varepsilon > 0$ such that 
$$ < a(x) \xi , \xi >  \; \geq \varepsilon \| \xi \|^2 $$
for all $x \in \RR^d .$ Classical results on lower bounds for transition densities of diffusions 
(see for instance Kusuoka and Stroock (1987)) imply that in this case any compact set of $\RR^d $ is petite.  
We cite the following recurrence conditions from Fort and Roberts (2005). Suppose there exist $M, \beta , \gamma > 0 $ and $ l < 2 $ such that 
$$ \sup_{x : \|x\| > M} \|x\|^{-(2 +l) } <x, a(x) x > = \beta , \quad  \sup_{x : \|x\| > M} \|x\|^{-l } tr (a(x)) = \gamma , $$
$$ \sup_{x : \|x\| > M} \|x\|^{-l } < b(x) , x > = -r , \mbox{ for some } r > ( \gamma - \beta l ) /2 .$$
We choose
$$\kappa \in \; ] 0, l + \frac{2r - \gamma}{\beta} [ $$
and put $m = 2 - l + \kappa ,$ thus $2 - m = l - \kappa .$ 
Let $V(x) = \|x\|^{m} $ outside a compact set. Then $ \sup_{x : \|x\| > M} {\cal A} V (x) < \infty ,$ and standard calculus shows that
for all $\|x\| > M,$
$$  {\cal A} V (x) \le m  \left( -r + \frac12 [\gamma + ( m - 2) \beta] \right) \frac{V(x)}{\|x\|^{2-l}}.$$
Then by our choice of $\kappa, $ $\tilde r :=   r -  \frac12 [\gamma + ( m - 2) \beta] > 0 .$ 
Hence for $\|x\| > M,$
$$  {\cal A} V (x) \le - \Phi  \circ V (x) , $$
where 
$$ \Phi (x) = m \tilde r \,  x^{1 - \alpha }, \quad \mbox{ with } \alpha = \frac{2 - l}{m } < 1 . $$
Hence we get polynomial moments of regeneration times up to the order $ m / (2- l ) = 1 + \kappa / ( 2 - l ) .$

\subsection{Solutions to SDE's driven by a jump noise}
This chapter is inspired by a recent work of Kulik (2009) on exponential ergodicity for solutions to SDE's driven by a jump noise. More precisely, consider the solution of the following stochastic differential equation on $\RR^d $ driven by a jump
noise
\begin{equation}\label{eq:jumpnoiseSDE}
d X_t = b (X_t) dt + \int_{\| u \| \le 1} c( X_{s-}, u) \tilde \mu (dt, du)  + \int_{ \| u\| > 1 } c( X_{s-}, u) \mu (dt, du) .
\end{equation}
Here $\mu$ is a Poisson random measure (PRM) on $ \RR_+ \times \RR^q ,$ having compensator $ \hat \mu (dt, du) = dt \nu (du),$
and $ \tilde \mu (dt, du) = \mu (dt, du) - dt \nu (du) $ denotes the compensated PRM. We follow Kulik (2009) and 
impose the following conditions on the 
coefficients $b$ and $c.$ The drift function $b$ belongs to $C^1 ( \RR^d, \RR^d )$ and satisfies a linear growth condition. 
The jump rate $c(x,u)$ is one times continuously differentiable with respect
to $x.$ Moreover, 
$$ \| c(x,u) - c( y,u) \| \le K ( 1 + \| u \| ) \| x-y\| , \quad \| c(x,u) \| \le \psi (x) \| u\| , \; x,y \in \RR^d , u \in \RR^q ,$$
where $K$ is some constant and where $\psi : \RR^d \to \RR_+ $ satisfies a linear growth condition. Finally we impose a 
moment condition on the L\'evy measure $\nu.$ For all $R > 0 ,$ 
$$ \int \sup_{ x : \| x \| \le R }  \left( \| c (x,u) \| + \| \nabla_x c(x,u) \| \right) \nu (du) < + \infty .$$
Then for any fixed $x \in \RR^d,$ there exists a unique strong solution $X_t$ to (\ref{eq:jumpnoiseSDE}), which is a strong Markov process, having c\`adl\`ag trajectories. 

We quote sufficient conditions implying that compact sets are petite from Kulik (2009). 
For this sake, we have to introduce some notation. 
Let ${\cal S}^q = \{ v \in \RR^q : \| v \| = 1\} $ be the unit sphere in $\RR^q.$ 
For any $w \in {\cal S}^q$ and for any $ \varrho \in ]0,1[, $ let $V_+ (w, \varrho) = \{ y \in \RR^q : \; < y, w > \; \geq \varrho \| y \| \} , $ 
and $ V ( w, \varrho) = \{ y \in \RR^q : | < y,w>| \; \geq \varrho \| y \| \} .$ Then Kulik (2009) obtains the following result. 

\begin{prop}[Kulik 2009]\label{prop:kulik}
Suppose that the following assumptions hold.
\begin{enumerate}
\item {Cone condition:} 
For every $ w \in {\cal S}^q,$ there exists $ \varrho \in ]0, 1[, $ such that for every $\delta > 0,$ 
$$ \nu \left( V(w,\varrho) \cap \{ u : \| u \| \le \delta \} \right) > 0 .$$
\item {Non-degeneracy condition:} 
There exists a point  $x_* \in \RR^d $ and a neighbourhood $O_* $ of $x_*$ such that $ c(x,u) = \chi (x) u + \delta (x,u) ,$ for all $x \in O_*,$ and 
$$ \| \delta (x_* , u ) \| + \| \nabla_x \delta (x_* , u) \| = o ( \| u\| ) , $$
as $ \| u \| \to 0 . $ Moreover, the functions $ \tilde b (.) = b (.) - \int_{ \| u \| \le 1 } c (., u) \nu (du) $ and $\chi $ are one 
times continuously differentiable and satisfy the joint non-degeneracy condition 
$$ rank \left( \nabla \tilde b (x_*) \chi (x_*) - \nabla \chi (x_* ) \tilde b (x_*) \right) = d .$$ 
\item {Support condition:}
For any $R > 0$ there exists $t$ such that for all $x$ with $\|x\| \le R, $ 
$$ x_* \in supp P_t (x, \cdot ).$$
\end{enumerate}
If the above conditions hold, then any compact set is petite. 
\end{prop}

\begin{rem}
1. In the one-dimensional case $ d= q = 1,$ the above conditions can be stated in a simpler way. For example, condition 1. can then be written as follows: For all $\delta > 0, $ $ \nu ( u : 0 < \| u \| \le \delta ) > 0 .$ \\
2. Simon (2000), Theorem I, gives a sufficient condition for condition 3. above to hold, see also Proposition 4.7 in Kulik (2009).
\end{rem}

\begin{proof}
Theorem 1.3, Proposition 4.3 and Proposition 4.4 of Kulik (2009) show that under the above conditions, the following Dobrushin condition holds: For all $R> 0 ,$ there exists $ t^* = t^* (R) $ such that 
\begin{equation}\label{eq:ld} \inf_{ x,y : \| x\|, \|y\| \le R } \int \left[ P_{t^* } (x, \cdot) \wedge  P_{t^* } (y, \cdot) \right] (dz) > 0 ,
\end{equation}
where for any two probability measures $P$ and $Q,$ 
$$ [ P \wedge Q ] (dz) := \left( \frac{dP}{d (P+Q)} (z) \wedge \frac{dQ}{d (P+Q)} (z) \right) (P+Q) (dz) .$$
From this the claim follows since (\ref{eq:ld}) implies that any compact set is a petite set. 
\end{proof}

It remains to give conditions that are sufficient for the recurrence condition (\ref{eq:driftcond}), (\ref{nice}) respectively. There is a wide
range of possible conditions and in what follows we restrict attention to a particular sufficient condition which is stated in the same spirit
as the conditions of Proposition 4.1 of Kulik (2009). 

\begin{prop}
Suppose that the conditions of Proposition \ref{prop:kulik} hold. Suppose moreover that there exist $M , \gamma > 0 $ and $0 < l < 1 $ such that  
\begin{enumerate}
\item Moment-condition: 
There exists $ m \geq 1  $ such that 
$ \int_{ \| u \| \geq 1 } \| u \|^m \nu (du) < \infty .$
\item Moderate jumps: 
The function $c$ can be decomposed into $ c= c_1 + c_2 $ such that 
\begin{enumerate}
\item
$ \| c_1 (x,u) \| \le \gamma \| x\|^{l}  \| u\| , u \in \RR^q, \| x\| > M .$ 
\item
$ \| x + c_2 (x,u) \| \le \| x\| , $ $\|u \| > 1, $ $\|x\| > M ,$ and $ c_2 (\cdot, u) = 0 , $ if $\| u\| \le 1 .$  
\end{enumerate}
\item Drift-condition:
$ \sup_{ x : \| x\| > M } \| x\|^{-(1+l) } < b(x) , x> = -r , $ for some constant $r $ satisfying $r >  2 \gamma \int_{ \| u\| > 1 } \|u\|^m \nu (du) .$ 
\end{enumerate}

Then there exists $M_0 \geq M$ such that (\ref{nice}) holds with $B = \{ x : \|x\| \le M_0 \},$ $B$ petite, $V(x) = \|x\|^m $ and 
$\Phi (x) = c x^{1- \alpha } ,$ where $ \alpha = \frac{1-l}{m} < 1 .$   
\end{prop}

\begin{proof}
We use the drift condition for the generator defined for all functions 
$F$ in the extended domain of the generator
\begin{multline*}
 {\cal A} F(x) = < b(x), \nabla F (x) > + \int_{\RR^q} ( F( x + c(x,u)) - F(x) - 1_{ \{ \| u \| \le 1 \}} \\
< \nabla F (x), c(x,u) > ) \nu (du) .
\end{multline*}
Applying this to $V (x) = \| x\|^m$ yields for all $\| x\| > M,$
\begin{eqnarray}\label{eq:uff}
&&{\cal A} V(x) = m < b(x), x >  \| x\|^{m-2} + \int_{\| u \| > 1 } \left( \| x + c(x,u) \|^m - \|x\|^m  \right) \nu (du) \nonumber \\
& & + \int_{\| u \| \le  1 } \left( \| x + c(x,u) \|^m - \|x\|^m- m  < x, c(x,u) > \|x\|^{m-2}\right) \nu (du) \nonumber \\
&&\le  -m\cdot r  \| x\|^{m-1 +l} + \int_{\| u \| > 1 } \left( \| x + c(x,u) \|^m - \|x\|^m  \right) \nu (du) \nonumber \\
& & + \int_{\| u \| \le  1 } \left( \| x + c(x,u) \|^m - \|x\|^m- m  < x, c(x,u) > \|x\|^{m-2}\right) \nu (du).
\end{eqnarray}
We start with the term in the last line. By Taylor's formula, writing $h = c(x,u) = c_1 (x,u ) ,$ since $\| u \| \le 1,$ we certainly have that 
\begin{eqnarray*}
&& \Big| \;  \| x + c(x,u) \|^m - \|x\|^m- m  < x, c(x,u) > \|x\|^{m-2} \; \Big| \\
&&\le \frac12 \sup_{ y \in ]x, x+h[ } | < h , \nabla^2 V (y) h > |\\
&& \le \frac12 m \left[ 1 + |m-2| \right] \| h\|^2 \! \!  \sup_{ y \in ]x, x+h[ } \| y \|^{m - 2 } .
\end{eqnarray*}
Here, $]x, x+h[$ denotes the $d-$dimensional interval $ ]x_1, x_1 + h_1 [ \times \ldots \times  ]x_d, x_d + h_d [ .$ 

Applying condition 2. (a) to $ h = c_1(x,u),$ where $\| u \| \le 1,$ yields 
$$ \| h \|^2 \le \gamma^2  \|x\|^{2l} \|u\|^2 .$$
If $ m - 2 > 0,$ we choose $M_0\geq M $  such that 
$ ( 1 + \gamma M_0^{l - 1} )^{m -1}  \le 2$ (recall that $l < 1 $). Then we obtain
\begin{eqnarray*}
 \sup_{ y \in ]x, x+h[ } \| y \|^{m - 2 }  = \| x + h \|^{m-2} &\le&
  \|x\|^{m-2} \left[ 1 + \gamma \|x\|^{l  - 1} \right]^{m-2} \\
  & \le &  \|x\|^{m-2}  \left[ 1 + \gamma M_0^{l - 1 } \right]^{m-2} \\
  &\le & 2  \|x\|^{m-2}  .
\end{eqnarray*}  
If $ m < 2,$ we can proceed similarly, 
\begin{eqnarray*}
 \sup_{ y \in ]x, x+h[ } \| y \|^{m - 2 }   &\le&
  \|x\|^{m-2} \left[ 1 -  \gamma \|x\|^{l  - 1} \right]^{m-2} \\
  & \le &  \|x\|^{m-2}  \left[ 1 - \gamma M_0^{l - 1 } \right]^{m-2} \\
  &\le&2 \|x\|^{m-2} ,
\end{eqnarray*}   
where we choose $M_0$ such that  $ ( 1 - \gamma M_0^{l - 1} )^{m-2} \le 2 . $ 

As a consequence, for any $\|x\| \geq M_0,$ the last line of (\ref{eq:uff}) is bounded from above by 
\begin{equation}\label{eq:lastline}
 m \left(   [ 1 + | m-2|] \gamma^2 \int_{\| u \| \le 1 } \|u \|^2 \nu (du) \right) \|x\|^{ m-2 + 2 l} 
 \le C \; M_0^{l-1}  \|x\|^{ m-1 +  l} ,
\end{equation}  
since $ \| x\|^{l-1} \le M_0^{ l-1} .$ Here, $M_0^{l-1} \to 0$ as $M_0 \to \infty,$ and $C$ is some constant. Hence the last term of (\ref{eq:uff}) will be neglectable for our purposes.  

Concerning the first jump term in (\ref{eq:uff}) we proceed as Kulik (2009), proof of Proposition 4.1: For $\|u\| > 1,$ using condition 2. (b), we have 
\begin{multline*}
 \| x + c(x,u) \|^m - \| x\|^m \le \| x + c(x,u) \|^m - \| x + c_2 (x,u) \|^m \\
= \| x(u) + c_1 (x,u) \|^m - \|x(u) \|^m,\end{multline*}
where $x(u) = x + c_2 (x,u) ,$ 
and then, applying Taylor's formula, 
$$  \| x(u) + c_1 (x,u) \|^m - \|x(u) \|^m \le m  \| c_1 (x,u)\| \sup_{ y \in ] x(u) , x(u) + c_1 (x,u) [ } \| y \|^{m-1} .$$
Now, since $m \geq 1,$ we argue as before and obtain, using successively condition 2. (a) and 2. (b) and $\|u\| > 1,$ 
\begin{eqnarray*}
m  \| c_1 (x,u)\| \sup_{ y \in ] x(u) , x(u) + c_1 (x,u) [ } \| y \|^{m-1} &\le& 
 m \gamma \|x\|^l \|u\| \left( \| x(u) \| + \gamma \| x\|^l \|u\| \right)^{m-1} \\
 &\le & m \gamma \|x\|^l \|u\| \left(  \|x\|  + \gamma \| x\|^l \|u\| \right)^{m-1}\\
&\le & m \gamma \|x\|^{m- 1 +l} \|u\|^m \left( 1 + \gamma M_0^{l-1} \right)^{m-1} \\
&\le & 2 m \gamma \|x\|^{m- 1 +l} \|u\|^m ,
\end{eqnarray*}
by the choice of $M_0 .$ As a consequence, the first jump term in (\ref{eq:uff}) can be upper bounded as follows:
$$  \int_{\| u \| > 1 } \left( \| x + c(x,u) \|^m - \|x\|^m  \right) \nu (du)  \le m \|x\|^{m-1+l} \left[ 2 \gamma \int_{\| u\| > 1 } \| u\|^m \nu (du) \right] .$$
Collecting all the above results, we finally obtain that for all $\|x\| \geq M_0,$ 
$$ {\cal A} V(x) \le m \left( - r + 2 \gamma \int_{ \| u\| > 1 } \|u\|^m \nu (du) + C M_0^{l-1} \right) \frac{V(x)}{\|x\|^{1- l}}.$$
By condition 3., for $M_0 $ sufficiently large, $ - r + 2 \gamma \int_{ \| u\| > 1 } \|u\|^m \nu (du) + C M_0^{l-1}  < 0 $
eventually, 
and this implies the assertion. 
\end{proof}

\section{Appendix}
For the convenience of the reader we give in this section a Fuk-Nagaev inequality for sums of two-dependent identically distributed centred random variables admitting a moment of order $p.$ This inequality is the key tool to our deviation inequalities, and its proof can be found in the book of Rio (2000). 

Let $X_1, X_2, \ldots $ be centred identically distributed random variables which are two-dependent and such that $ E (|X_1|^p ) < \infty $ for some $ p \geq  2 .$ Put $ S_k := X_1 + \ldots + X_k .$ Then the following deviation inequality holds.

\begin{theo}\label{theo:fuknagaev}
For any $ \lambda > 0 $ we have that
$$ P \left( \sup_{ k \le n } |S_k| \geq 4 \lambda \right) \le C ( p) \left( \sigma^{ 2 (p-1)} \lambda^{ - 2 (p-1)} n^{p-1} + m_p \, n \; \lambda^{-p } \right) ,$$
where $ \sigma^2 = Var (X_1)$ and $m_p = E(|X_1|^p ).$  
\end{theo}

\begin{proof}
We use the Fuk-Nagaev inequality presented in Rio (2000), Th\'eor\`eme 6.2. First of all, since the variables are two-dependent, we certainly have the upper bound on the $\alpha -$mixing coefficients $\alpha_n = \sup_{k \geq n } \alpha ( \sigma (X_1) , \sigma (X_{k+1})) :$
$$ \alpha_0 = \frac12, \alpha_1 \le \frac12 , \alpha_n \equiv 0 \mbox{ for all } n \geq 2 .$$
Hence, using the notation (1.21) of Rio (2000), we can upper bound
$$ \alpha^{- 1 } (u) \le 2 \;  1_{[0, \frac12 [} (u) .$$ 
As a consequence the expression $ R(u) $ of Th\'eor\`eme 6.2 of Rio (2000) is given as 
$$ R(u) \le 2 Q(u) 1_{[0, \frac12 [} (u) \le 2 Q (u) , $$
where $ Q(u) = \inf \{ x : H_{X_1} (x) \le u \} $ is the quantile of $|X_1|$, $ H_{X_1} (t) = P ( | X_1| > t ).$
But since $X_1$ admits a $p-$th moment, we certainly have that
$$ H_{X_1} (t) \le m_p \; t^{- p },$$
by Markov's inequality (recall that $m_p = E( |X_1|^p )$). Since $Q$ is the generalised inverse function of $H_{X_1} ,$ this implies that 
$$ Q(u) \le m_p^{1/p } \; u^{ - 1/p} ,$$ 
and this in turn leads to 
$$ H(u) = R^{-1} (u) \le 2^p \; m_p u^{ -p}
 . $$
Now we can apply (6.5) of Rio (2000). First of all notice that by the two-dependency structure
$$ s_n^2 = \sum_{i= 1}^n \sum_{j = 1}^n |Cov (X_i, X_j) | \le 3 n \sigma^2 .$$ 
Thus we obtain, for any $r \geq 1,$  
\begin{eqnarray*}
P\left( \sup_{k \le n } |S_k| \geq 4 \lambda \right) &\le& 4 \left( 1 + \frac{ \lambda^2}{r s_n^2} \right)^{ - r/2} + 4 n \lambda^{-1} \int_0^{ H( \lambda/r)} Q(u) du \\
&\le &4 \left( 1 + \frac{ \lambda^2}{r s_n^2} \right)^{ - r/2} + 4 n \lambda^{-1} \int_0^{ H( \lambda/r)} m_p^{ 1/p } u^{- 1/p}  du \\
&=&4 \left( 1 + \frac{ \lambda^2}{r s_n^2} \right)^{ - r/2} + \frac{4 n \lambda^{-1}\,  m_p^{1/p} }{ ( 1 - \frac1p)} H(\lambda /r)^{ 1 - \frac1p } \\
&\le & 4 \left( 1 + \frac{ \lambda^2}{r s_n^2} \right)^{ - r/2} + C( p)  m_p \, n \, \lambda^{-p} \, r^{p-1} \\
&\le & 4 \left( 1 + \frac{ \lambda^2}{3 n r \sigma^2} \right)^{ - r/2} + C( p) \,  m_p \,  n \, \lambda^{-p} \, r^{p-1} .
\end{eqnarray*}
Here, $ C(p) $ is a constant depending only on $p.$ Now we choose $ r = 2 (p-1) .$ By assumption on $p,$ $r \geq 1 .$ Finally we get
$$ P\left( \sup_{k \le n } |S_k| \geq 4 \lambda \right) \le C ( p)   \left( \sigma^{ 2 (p-1)} \lambda^{ - 2 (p-1)} n^{p-1} +m_p \,  n \; \lambda^{-p } \right) .$$
\end{proof}

\end{document}